\documentclass[11pt]{amsart}

\usepackage[a4paper,hmargin=3.5cm,vmargin=4cm]{geometry}
\usepackage{amsfonts,amssymb,amscd,amstext,mathrsfs,color}

\usepackage{graphicx}
\usepackage[dvips]{epsfig}

\usepackage{verbatim}
\usepackage{inputenc}
\usepackage{hyperref}

\usepackage{fancyhdr}
\input xy
\xyoption{all}

\usepackage{enumerate}

\usepackage{times}
\usepackage{enumerate}
\usepackage{titlesec}
\usepackage{mathrsfs}

\titleformat{\section}
{\filcenter\bfseries\large} {\thesection{.}}{0.2cm}{}
\titleformat{\subsection}[runin]
{\bfseries} {\thesubsection{.}}{0.15cm}{}[.]
\titleformat{\subsubsection}[runin]
{\em}{\thesubsubsection{.}}{0.15cm}{}[.]

\usepackage[up,bf]{caption}

\newtheorem{theorem}{Theorem}
\newtheorem{proposition}[theorem]{Proposition}

\newtheorem{lemma}[theorem]{Lemma}
\newtheorem{corollary}[theorem]{Corollary}

\theoremstyle{definition}
\newtheorem{definition}[theorem]{Definition}
\newtheorem{remark}[theorem]{Remark}

\newtheorem{problem}[theorem]{Problem}
\newtheorem{example}[theorem]{Example}

\numberwithin{equation}{section}
\numberwithin{figure}{section}

\newcommand\Ascr{\mathscr{A}}

\newcommand\Cscr{\mathscr{C}}

\newcommand\Escr{\mathscr{E}}

\newcommand\Lscr{\mathscr{L}}
\newcommand\Oscr{\mathscr{O}}

\newcommand\C{\mathbb{C}}

\newcommand\CP{\mathbb{CP}}

\newcommand\N{\mathbb{N}}
\renewcommand\P{\mathbb{P}}
\newcommand\R{\mathbb{R}}

\newcommand\Z{\mathbb{Z}}

\newcommand\igot{\mathfrak{i}}

\renewcommand\igot{\mathfrak{i}}

\newcommand\ggot{\mathfrak{g}}

\renewcommand\imath{\igot}

\newcommand\hra{\hookrightarrow}
\newcommand\lra{\longrightarrow}

\newcommand\longhookrightarrow{\ensuremath{\lhook\joinrel\relbar\joinrel\rightarrow}}

\newcommand\wt{\widetilde}

\newcommand\Oscrnc{\Oscr_\mathrm{nc}}

\usepackage{color}

\begin{document}

\fancyhead[LO]{Holomorphic Legendrian curves in projectivised cotangent bundles }
\fancyhead[RE]{F.\ Forstneri\v c and F.\ L{\'a}russon} 
\fancyhead[RO,LE]{\thepage}

\thispagestyle{empty}

\vspace*{1cm}
\begin{center}
{\bf\LARGE Holomorphic Legendrian curves \\ in projectivised cotangent bundles}

\vspace*{0.5cm}

{\large\bf  Franc Forstneri{\v c} and Finnur L{\'a}russon} 
\end{center}


\vspace*{1cm}

\begin{quote}
{\small
\noindent {\bf Abstract}\hspace*{0.1cm}
We study holomorphic Legendrian curves in the standard complex contact structure on the projectivised 
cotangent bundle $X=\P(T^*Z)$ of a complex manifold $Z$  of dimension at least $2$. We provide a detailed
analysis of Legendrian curves degenerating to vertical curves and obtain several approximation and general 
position theorems. In particular, we prove that any vertical holomorphic curve $M\to X$ from a compact 
bordered Riemann surface $M$  can be deformed to a horizontal Legendrian curve by an arbitrarily small 
deformation. A similar result is proved in the parametric setting, provided that all vertical curves under consideration are 
nondegenerate. Stronger results are obtained when the base $Z$ is an Oka manifold or a Stein manifold with 
the density property.  Finally, we establish basic and 1-parametric h-principles for holomorphic Legendrian 
curves in $X$.

\vspace*{0.2cm}

\noindent{\bf Keywords}\hspace*{0.1cm} complex contact manifold, projectivised cotangent bundle, Legendrian curve, Riemann surface, Stein manifold, Oka principle, h-principle

\vspace*{0.1cm}

\noindent{\bf MSC (2010)}\hspace*{0.1cm} 53D10; 32E30, 32H02, 37J55}

\vspace*{0.1cm}
\noindent{\bf Date \rm 25 September 2018. Minor edits on February 6, 2022.}

\end{quote}


\section{Introduction} 
\label{sec:intro}

A {\em complex contact manifold} is a pair $(X,\xi)$, where $X$ is a complex manifold of 
(necessarily) odd dimension $2n+1\ge 3$ and $\xi$ is a {\em contact subbundle} of 
the holomorphic tangent bundle $TX$, that is, a maximally nonintegrable
holomorphic hyperplane subbundle  of $TX$.  More precisely, the {\em O'Neill tensor}
$O: \xi\times \xi \to TX/\xi$, $(v,w)\mapsto [v,w] \!\!\mod \xi$, is nondegenerate. 
Note that $\xi=\ker\alpha$ where $\alpha$  is a holomorphic $1$-form on $X$ with values in the 
normal line bundle $L=TX/\xi$ realising the quotient projection: 
\[  
	0 \lra \xi \longhookrightarrow TX \stackrel{\alpha}{\lra} L \lra 0.
\] 
The contact condition is equivalent to $\alpha\wedge (d\alpha)^n\ne 0$ at every point of $X$.
As in the real case, a complex contact structure has no local invariants and is locally at any point 
contactomorphic to the standard contact structure on $\C^{2n+1}$ given by the $1$-form
\begin{equation}\label{eq:standard}
	\alpha_{\mathrm{std}} = dz+\sum_{j=1}^n x_j dy_j,
\end{equation}
where $(x,y,z)$ are complex coordinates on $\C^{2n+1}$.
(See Darboux \cite{Darboux1882CRAS}, Moser \cite{Moser1965TAMS}, and 
Geiges  \cite[p.~67]{Geiges2008} for the real case, 
and \cite[Theorem A.2]{AlarconForstnericLopez2017CM} for the holomorphic case.)
A more complete introduction to the subject can be found in the papers
by LeBrun \cite{LeBrun1995IJM},  LeBrun and Salamon \cite{LeBrunSalamon1994IM},
the survey by Beauville \cite{Beauville2011}, and the recent paper by
Alarc\'on and the first named author   \cite{AlarconForstneric2017IMRN}
where a contact neighbourhood theorem has been established 
for any  immersed noncompact Legendrian curve in an arbitrary complex contact manifold.

It is well known that for any complex manifold $Z$ of complex dimension at least two, 
the projectivised holomorphic cotangent bundle  $\P(T^*Z)$ admits a natural complex contact structure.
It is conjectured that, among compact projective manifolds, this is one of only two types of 
manifolds admitting a complex contact structure; see Beauville \cite[Conjecture 6]{Beauville2011}.
The second type is contact Fano manifolds which arise as the unique closed orbit of 
the adjoint action of a simple complex Lie group $G$ on the 
projectivisation $\P(\ggot)$ of the Lie algebra $\ggot$ of $G$. 
We shall not consider manifolds of the second type in this paper. 

Let us  recall the definition of the contact structure on a projectivised cotangent bundle $X=\P(T^*Z)$;
cf.\ Geiges \cite[Lemmas 1.2.2 and 1.2.3]{Geiges2008} for the smooth case. 

A {\em contact element} in a complex manifold $Z$ is a complex hyperplane $\Sigma\subset T_z Z$ 
in a tangent space to $Z$. Such $\Sigma$ is the kernel of a covector $0\ne a \in T^*_z Z$ 
which is  determined up to a nonzero constant multiple, and hence by 
the complex line $\C a \subset T^*_z Z$. This shows that the space of contact elements of $Z$ 
can be naturally identified with the projectivised cotangent bundle $X=\P(T^*Z)$, 
a complex manifold of dimension $2n+1$ if $\dim Z=n+1$. 
Let $\pi\colon T^*Z\to Z$ denote the base projection. Then, $\pi$ induces the projection
$\pi\colon X\to Z$ which is a holomorphic fibre bundle with fibre $\CP^n$.
The cotangent bundle $T^*Z$ carries the tautological holomorphic $1$-form  $\eta$, 
called the {\em Liouville form} (see e.g.\ \cite[p.\ 19]{Geiges2008}) whose value on a tangent vector 
$v\in T_a(T^*Z)$ at a point $a\in T_z^* Z$ equals 
\begin{equation}\label{eq:eta0}
	\langle \eta,v\rangle = \langle a, (d\pi)_a(v)\rangle.  
\end{equation}
Here, $\langle \eta,v\rangle$ is the value of  $\eta$ on the vector $v$, and 
$\langle a, (d\pi)_a(v)\rangle$ is the value of the covector 
$a\in T_z^* Z$ on the vector $(d\pi)_a(v)\in T_z Z$.
The $1$-form $\eta$ is given in any local holomorphic coordinates 
$(z_0,\ldots,z_n)$ on $Z$ and associated fibre coordinates $(\zeta_0,\ldots,\zeta_n)$ on $T_z^*Z$ by 
\begin{equation}\label{eq:eta}
	 \eta= \sum_{j=0}^n \zeta_j dz_j. 
\end{equation}
Taking $(\zeta_0,\ldots,\zeta_n)$ as  homogeneous coordinates on $\P(T_z^*Z)$ 
we get the holomorphic contact structure $\xi=\ker\eta$ on $X=\P(T^*Z)$. 
On the affine chart $\zeta_j=1$ we have
\begin{equation}\label{eq:affine}
	\xi= \ker\bigl(dz_j + \sum_{i\ne j} \zeta_idz_i \bigr);
\end{equation}
this agrees with the standard contact structure \eqref{eq:standard} on $\C^{2n+1}$.
Note that $\P(T^*\C^{n+1})\cong \C^{n+1}\times \CP^{n}$ 
is covered by $n+1$ affine contact charts \eqref{eq:affine}.

We see from \eqref{eq:eta0} that the contact element $\xi_a$ at any point 
$a \in X$ is the unique complex hyperplane in $T_a X$ whose  projection to $T_{\pi(a)}Z$
is the hyperplane $\ker(a) \subset T_{\pi(a)}Z$. Note that the {\em vertical tangent space}
to $X$ at any point $a\in X$ (i.e., the tangent space to the fibre of the projection $X\to Z$)
is a subspace of the contact element $\xi_a$.

Clearly, $X$ is compact and simply connected if and only if $Z$ is.  Then the complex contact structure on $X$ that we have described is in fact the only one, up to a biholomorphism of $X$ isotopic to the identity, by a theorem of LeBrun and Salamon \cite[Proposition 2.2]{LeBrunSalamon1994IM}.

Assume now that $(X,\xi)$ is a complex contact manifold and $M$ is a smooth
manifold. A smooth map $f\colon M\to X$ is said to be 
{\em isotropic}, or an {\em integral submanifold} of the contact structure $\xi$, if 
\[
	df_x (T_x M)\subset \xi_{f(x)} \quad \text{for all $x\in M$}.
\]
This is equivalent to $f^*\alpha=0$ for any local $1$-form $\alpha$ with $\ker\alpha=\xi$.
If $f$ is an immersion at a generic point of $M$, then the contact property of $\xi$ implies 
$\dim_\R M\le 2n$ where $\dim_\C X=2n+1$. Immersed isotropic submanifolds of maximal dimension 
are called {\em Legendrian} and are necessarily complex submanifolds of $X$
(see \cite[Lemma 5.1]{AlarconForstneric2017IMRN}). For this reason, we shall restrict
attention to holomorphic Legendrian maps $M\to X$ from complex manifolds $M$.
In particular, isotropic complex curves in a contact threefold  are Legendrian; this is the case
for $X= \P(T^*Z)$ when $Z$ is a complex surface. 

We denote by $\Lscr(M,X)$ the space of all isotropic holomorphic maps $M\to X$.  
If $M$ is a Riemann surface, elements of $\Lscr(M,X)$ will be called {\em holomorphic Legendrian curves} 
in $X$ regardless of the dimension of $X$. 
Note that every map into a fibre of the projection
$\pi\colon X= \P(T^*Z) \to Z$ is isotropic; we call such maps \emph{vertical}. 
A holomorphic map $f\in \Oscr(M,X)$ is called {\em horizontal}
if the projection $g=\pi\circ f \colon M\to Z$ has discrete fibres;
equivalently, if $g$ is nonconstant on any complex curve in $M$.
If $M$ is a Riemann surface then any map $M\to X$ is either horizontal or vertical.
A more precise description of isotropic submanifolds in $\P(T^*Z)$ is given in Sect.\ \ref{sec:basic}.

The first main focus of the present paper is the analysis of holomorphic Legendrian curves in 
$X=\P(T^*Z)$ in a neighbourhood of vertical curves. We are particularly interested in the
problem of approximating and deforming vertical Legendrian curves
to horizontal ones. In general this is impossible.
For example, if the manifold $Z$ is Brody hyperbolic (i.e., it admits no nonconstant 
holomorphic images of $\C$), then every nonconstant holomorphic curve $\C\to X$ is vertical;
there are plenty of those since the fibre of the projection $\pi:X\to Z$ is $\CP^n$.
Similarly, if $M$ is a compact complex manifold, then a vertical holomorphic map $M\to X_z=\pi^{-1}(z)$
$(z\in Z)$ cannot be approximated by holomorphic maps with nonconstant projection to $Z$ in view
of the maximum principle.

The situation is rather different if $M$ is a compact bordered Riemann surface.
In this case, every vertical holomorphic curve $M\to X=\P(T^*Z)$ 
can be deformed to a horizontal Legendrian curve by an arbitrarily small deformation
(see Proposition \ref{prop:approx1}). 
We actually prove a considerably stronger parametric general position theorem (see  Theorem \ref{th:B}) 
which says that, for a compact locally contractible parameter space $P$ and a 
compact bordered Riemann surface $M$, a continuous family
of holomorphic Legendrian curves $f_p\colon M\to X$, $p\in P$, can be approximated 
by a continuous family of horizontal holomorphic 
Legendrian curves $M\to X$ provided that every vertical map in the
family $f_p$ has nondegenerate vertical component $h_p\colon M\to \P(T^*_z Z)\cong \CP^n$, 
in the sense that $h_p(M)$ is not contained in any proper projective subspace 
of $\CP^n$.  Moreover, the deformation may be kept fixed on any compact subset $Q\subset P$ 
such that the curve $f_p$ is horizontal for every $p\in Q$. 

We do not know whether Theorem \ref{th:B} holds in the presence of degenerate vertical maps in 
the family $\{f_p\}_{p\in P}$.  The analytic difficulties that appear 
in the degenerate case are reminiscent of those in the theory of conformal minimal 
surfaces in Euclidean spaces $\R^n$, $n\ge 3$, where {\em flat surfaces} (i.e., those lying in 
affine $2$-planes in $\R^n$) seem to be singular points of the space of all conformal minimal surfaces 
(see \cite{AlarconForstnericLopez2019JGEA,ForstnericLarusson2019CAG}).

Assume now that $M$ is an open Riemann surface. We have already remarked
that it is in general impossible to approximate a vertical holomorphic map $M\to X=\P(T^*Z)$ 
by horizontal ones. However, assuming that $Z$ is an Oka manifold
(see \cite[Sect.\ 5.4]{Forstneric2017E} for the definition and main properties
of this class of complex manifolds), 
we show that every vertical holomorphic map can be approximated uniformly on compacts
in $M$ by horizontal Legendrian maps (see Corollary \ref{cor:approx2}).
Furthermore, if $Z$ is a Stein manifold with the density property 
(see \cite[Definition 4.10.1]{Forstneric2017E} and note that
every such manifold is Oka, cf.\  \cite[Proposition 5.6.23]{Forstneric2017E}), then $\P(T^*Z)$  
admits plenty of {\em proper} holomorphic Legendrian curves from any open Riemann surface
(see Corollary \ref{cor:approx2}), and also isotropic Stein submanifolds of higher dimension
when $\dim Z>2$ (see Sect.\ \ref{sec:Stein}).

Another interesting phenomenon is the following. Given a complex surface $Z$ and a family 
$f^\epsilon : M \to \P(T^*Z)$ of horizontal Legendrian holomorphic curves depending 
continuously on a real (or complex) parameter $\epsilon$, a branch point of the projection 
$g^\epsilon =\pi\circ f^\epsilon : M\to Z$ is stable (see Proposition \ref{prop:singularity}). 
In particular, if $g^\epsilon :M\to Z$ is a continuous family of nonconstant holomorphic maps
such that $g^\epsilon $ is an immersion for $\epsilon\ne 0$ but $g^0$ 
has a branch point, then the family of their Legendrian liftings $M\to \P(T^*Z)$
is discontinuous at $\epsilon=0$. 
This causes considerable technical difficulties in the proofs of our main results. 
When constructing horizontal Legendrian approximations of vertical curves (see Proposition \ref{prop:approx1}, 
Remark \ref{rem:approx1}, and Theorem \ref{th:B}), we pay special attention to preserving their fibre 
components, thereby ensuring the existence of a homotopy of Legendrian curves connecting the original 
vertical curves to their horizontal approximants.

Building on our analytic results, we establish basic and 1-parametric h-principles for closed holomorphic Legendrian curves in the projectivisation $X$ of the cotangent bundle of a complex manifold $Z$ (see Theorem \ref{th:better-h-principle}).  By a {\em closed} holomorphic curve in $X$, we mean the germ of a holomorphic map into $X$ from an open neighbourhood of a compact bordered Riemann surface $M$, sitting as a smoothly bounded compact domain in an ambient surface.  

We prove the basic h-principle that {\em every closed holomorphic curve $f:M\to X$ can be deformed to a closed holomorphic Legendrian curve} $h:M\to X$ (see Theorem \ref{th:basic-h-principle}).  The basic h-principle also holds for open and bordered curves.  If we restrict ourselves to {\em strong immersions} $f$, meaning that the composition $\pi\circ f$ by the projection $\pi:X\to Z$ is an immersion (so $f$ itself is also an immersion), then we can prove that {\em the Legendrian curve $h$ is unique up to homotopy}.  In fact, when $\dim Z\geq 3$, we prove the following 1-parametric h-principle:  A continuous path through strongly immersive closed holomorphic curves joining two Legendrian curves can be deformed, with fixed end points, to a path through Legendrian curves.  This fails in general if $\dim Z=2$.

\begin{remark}
Closely related to the class of projectivised cotangent bundles is another natural class of complex contact manifolds.  Let $\Sigma$ be a complex manifold and denote by $J^1 \Sigma$ the manifold of $1$-jets of holomorphic functions on $\Sigma$. A point in $J^1\Sigma$ is a triple $(q,p,z)$ with $q\in \Sigma$, $p\in T^*_q\Sigma$, and $z\in \C$, and $J^1\Sigma$ may be canonically identified with $T^*\Sigma\times \C$.  On $T^*\Sigma$ we have the tautological $1$-form $pdq$, whose differential $d(pdq)= dp \wedge dq$ is the standard symplectic form on $T^*\Sigma$.  The $1$-form $\alpha=dz+pdq$ defines the standard contact structure $\xi=\ker\alpha$ on $J^1\Sigma$.

Note that $(J^1\Sigma,\xi)$  is a domain in the projectivised cotangent bundle $X=\P(T^*Z)$ of the manifold $Z=\Sigma\times \C$. Indeed, using coordinates $(q,p)$ on $T^*\Sigma$ and $(z,\zeta)$ on $T^*\C=\C\times \C$, the Liouville form on $T^*Z$ is $\zeta dz+p dq$, which coincides with \eqref{eq:eta}.  Taking $\zeta=1$, we get $dz+pdq$, which is the contact form \eqref{eq:affine} on $J^1\Sigma$, considered as an affine chart in $\P(T^*Z)$.  See \cite[Sect.\ 6.5]{CieliebakEliashberg2012} for more details.
\qed\end{remark}

%
%
\section{Basic results on isotropic submanifolds of $\P(T^*Z)$} 
\label{sec:basic}

We begin this preparatory section with a more precise geometric description of isotropic complex 
submanifolds in the projectivised cotangent bundle $X=\P(T^*Z)$ of a complex manifold $Z$. 
The main observations 
are summarised in Proposition \ref{prop:summary} for future reference. 

Let $\pi\colon X\to Z$ denote the base projection; this is a holomorphic fibre bundle
with fibre $\CP^n$, where $\dim Z=n+1\ge 2$. 
Recall that $\Lscr(M,X)$ denotes the space of all isotropic holomorphic maps $M\to X$
from a connected complex manifold $M$. In Sect.\ \ref{sec:intro} we have already 
introduced the notion of a {\em vertical} map $M\to X$, i.e., a map with values in a fibre 
$X_{z}=\pi^{-1}(z)$ of $\pi$, and we have noted that every such map is isotropic.
On the other hand, a map $f\in \Oscr(M,X)$ is  {\em horizontal}
if the projection $g=\pi\circ f \colon M\to Z$ has discrete fibres;
equivalently, if $g$ is nonconstant on every complex curve in $M$.

Let us now consider the problem of finding isotropic liftings $f:M\to X$ of a holomorphic map
$g:M\to Z$. It follows from the definition of the contact structure $\xi$ on $X$
(see \eqref{eq:eta}) that such a lifting corresponds to a 
choice of a hyperplane field $\Sigma_x\subset T_{g(x)}Z$, depending holomorphically 
on $x\in M$, such that
\begin{equation}\label{eq:hyperplanes}
	dg_x(T_x M) \subset \Sigma_x,\qquad x\in M.
\end{equation}
Assume that $U$ is an open subset of $M$ whose image $g(U)\subset Z$ lies in 
a coordinate chart $V\subset Z$. In holomorphic coordinates $(z_0,\ldots,z_n)$ 
on $V$ and associated fibre coordinates $(\zeta_0,\ldots,\zeta_n)$ on $T^*V$, 
an isotropic lifting of a map $z:U\to V$ is determined by a holomorphic
map $\zeta=[\zeta_0:\ldots:\zeta_n]:U\to \CP^n$ satisfying the equation   (cf.\ \eqref{eq:eta})
\begin{equation}\label{eq:lifting}
	  \sum_{j=0}^n \zeta_j(x) dz_j(x)=0,\qquad x\in U.
\end{equation}
The space $\Escr$ of germs of holomorphic solutions $(\zeta_0,\ldots,\zeta_n)$ 
of \eqref{eq:lifting} is a coherent analytic sheaf on $M$ (the sheaf of relations between the 
differentials $dz_j$). Although our considerations are local, 
it is easily seen that the sheaf $\Escr$ is well defined globally on $M$.

If $g:M\to Z$ is an immersion, then the sheaf $\Escr$ is clearly locally free, i.e., the sheaf of sections
of a holomorphic vector bundle $E\to M$ of rank $n+1-m$ where $m=\dim M$. Note that $E$ 
is in a natural way a holomorphic vector subbundle of the pullback bundle $g^*(T^*Z)$,
called the {\em conormal bundle} of the immersion $g$. (Likewise,
$\Escr$ is called the {\em conormal sheaf} of the map $g$ even if it is not an immersion.)
The projectivised bundle $\P E$ is then a holomorphic subbundle of $g^*(\P\, T^*Z)$ with
fibre $\CP^{n-m}$, and isotropic liftings of $g$ are in bijective correspondence with sections
$M\to \P E$. Note that every nowhere vanishing holomorphic section of $\P E$
over a contractible Stein neighbourhood $U\subset M$ lifts to a nowhere vanishing 
holomorphic section of $E|_U$ by the Oka-Grauert principle.
This means that any isotropic lifting $U\to X$ of $z=g|_U:U\to Z$ is determined
by a holomorphic map $\zeta:U\to \C^{n+1}_*=\C^{n+1}\setminus\{0\}$ satisfying \eqref{eq:lifting}.
Of course this need not hold globally.

If $\dim M=n=\dim Z-1$ and $g:M\to Z$ is an immersion, then $E$ is a line bundle 
and $\P E$ is a rank $0$ bundle, so  $g$ lifts to a unique Legendrian 
holomorphic map $f:M\to X$. Geometrically, for any $x\in M$, the fibre component 
of the point $f(x)\in X$ is the complex hyperplane (contact element) 
$dg_x(T_x M)\subset T_{g(x)}Z$. More generally, a holomorphic map $g\colon M\to Z$ 
lifts to a unique Legendrian map  at all immersion points as explained above, 
but the lifting may have points of indeterminacy at branch points of $g$.

If $\dim M\le \dim Z-2$ then an immersion $g\colon M\to Z$ need 
not lift to an isotropic holomorphic map $f\in \Lscr(M,X)$ (since the bundle
$\P E\to M$ need not admit any global holomorphic sections), 
and if it does, the lifting need not be unique. 

Let us consider more closely the special case of particular interest to us when $M$ is a 
connected Riemann surface. In this case there are only two types of Legendrian maps 
$M\to X=\P(T^*Z)$, namely vertical and horizontal. 

If $\dim Z=2$, then every nonconstant holomorphic curve $g\colon M\to Z$ lifts to a unique 
horizontal holomorphic Legendrian curve $f\colon M\to X$, given in local coordinates by
\begin{equation}\label{eq:liftingcurve}
	F(x)=\bigl(g_0(x),g_1(x), [-(dg_1)_x : (dg_0)_x]\bigr),\qquad x\in M.  
\end{equation}
Note that the fibre component of $f(x)$ is well defined also at points where both differentials vanish.
The map $F$ is an immersion if and only if $[-dg_1:dg_0]$ is an immersion 
at every branch point of $g$.  Let us analyse this condition more explicitly. 
Using a holomorphic coordinate $x$ on $M$ 
centred at a branch point $x_0$ of $g$, we have $g_0'=a(x)x^j$ and
$g_1'=b(x) x^k$ where $a,b$ are holomorphic functions 
with $a(0)\ne 0,\, b(0)\neq 0$ and $j,k\geq 1$, say $j\leq k$.  
If $\lvert j-k\rvert=1$ then $[-g_1':g_0']=[-{\tfrac b a}x:1]$ is an immersion at $x=0$, 
but if $\lvert j-k\rvert\geq 2$ then it is not.  
If $j=k$ then $[-g_1':g_0']$ is an immersion if and only if $x_0$ is not a critical point of 
the function $b/a$. If $g\colon M\to Z$ is an immersion, then the lifting $F \colon M\to X$ 
is injective if and only if $g$ has distinct tangents at every self-intersection point of its image, that is, 
if $g(x_0)=g(x_1)=z\in Z$ for a pair of distinct points $x_0,x_1\in M$, 
then the complex lines in $T_z Z$ spanned by the vectors $g'(x_0)$ and $g'(x_1)$ are different.

Note that a field of hyperplanes satisfying \eqref{eq:hyperplanes}
always exists if $M$ is a connected open Riemann surface, $Z$ is an arbitrary complex
manifold of dimension $n+1\ge 2$, and $g\colon M\to Z$ is a
nonconstant holomorphic map. Indeed, the tangent cone of $g$ 
at each point is one-dimensional and the conormal sheaf $\Escr$ is locally free. 
This is obvious at an immersion point of $g$. At a branch point $x_0$ of $g$, 
choose a local holomorphic coordinate $x$ on $M$ with $x(x_0)=0$ and local holomorphic coordinates 
$z=(z_0,\ldots,z_n)$ on $Z$ around $g(0)\in Z$,  
with $z(g(0))=0$. Then, $g_j(x)=a_j x^{k_j}+ O(x^{k_j+1})$ with $a_j\ne 0$ and $k_j\ge 2$
for $j=0,\ldots,n$. Applying a linear change of the $z$ coordinates
if necessary, we may assume that $k_j>k_0$ for all $j=1,\ldots, n$.
Then, the tangent cone to $g$ at $x_0=0$ is the line $\C\times \{0\}^{n}$ 
(the $z_0$-axis). After dividing all components $g'_j(x)$ by $x^{k_0-1}$, the equation
\eqref{eq:lifting} is 
\[
	\zeta_0(x)(a_0 k_0 + O(x)) + \sum_{j=1}^n \zeta_j(x) (a_j k_j x^{k_j-k_0-1}+ O(x^{k_j-k_0})) =0.
\]
Since the coefficient $a_0 k_0 + O(x)$ is nonvanishing at $x=0$,
the equation can be solved on $\zeta_0$ for every choice of functions
$\zeta_1,\ldots,\zeta_n$ in a neighbourhood of $0$, thereby proving our claim.
Hence, $\Escr$ determines  a holomorphic hyperplane subbundle $E$ of $g^*(T^* Z)$.
Since every holomorphic vector bundle on an open Riemann surface is trivial
\cite[Theorem 5.3.1]{Forstneric2017E},
we infer that every nonconstant holomorphic map $g:M\to Z$ admits an isotropic
lifting $M\to X=\P(T^* Z)$, and the lifting is unique if and only if $\dim Z=2$.

We summarise the preceding discussion in the following proposition. 

%
%
\begin{proposition}\label{prop:summary}
Let $Z$ be a connected complex manifold of dimension $n+1\ge 2$, and let $X=\P(T^*Z)$ 
be the projectivised cotangent bundle of $Z$ with the standard complex contact structure.
Assume that $M$ is a connected complex manifold of dimension $k\in\{1,\ldots,n\}$.
\begin{enumerate}[\rm (a)]
\item If $k=n$, then every holomorphic immersion $g\colon M\to Z$ lifts to a unique
holomorphic Legendrian immersion $f\colon M\to X$.
\vspace{1mm}
\item If $k=n=1$ (i.e., $M$ is a Riemann surface and $Z$ is a complex surface),
then every nonconstant holomorphic map $g\colon M\to Z$ lifts to a unique
holomorphic Legendrian curve $f\colon M\to X$. 
The map $f$ is an immersion if and only if $[-dg_1:dg_0]:M\to \CP^1$ 
is an immersion at every branch point of $g$ (see \eqref{eq:liftingcurve}). 
If $g$ is an immersion with normal crossings then $f$ is an embedding.
\vspace{1mm}
\item If $1\le k<n$, then a holomorphic immersion $g\colon M\to Z$ lifts 
to a holomorphic isotropic map $f\colon M\to X$ if and only if there exists
a field of complex hyperplanes $\Sigma_x\subset T_{g(x)}Z$ depending holomorphically 
on $x\in M$ and satisfying \eqref{eq:hyperplanes}. Every isotropic lifting of $g$
corresponds to such a hyperplane field. 
\vspace{1mm}
\item If $M$ is an open Riemann surface, then every nonconstant holomorphic map $M\to Z$ 
lifts to a holomorphic Legendrian curve $M\to X$. The lifting is unique if and only if $\dim Z=2$,
and in this case the lifting exists also if $M$ is a compact Riemann surface without boundary.
\end{enumerate}
\end{proposition}

The following natural example deserves to be mentioned; compare with 
Bryant \cite[Theorem G]{Bryant1982JDG} concerning Legendrian embeddings into $\CP^3$.

%
%
\begin{example}\label{ex:CP2}
Every Riemann surface $M$ (either compact or open) admits a holomorphic Legendrian 
embedding into $\P(T^* \CP^2)$. 

Indeed, it is classical that every such $M$ admits a holomorphic immersion 
$M\to \CP^2$ with simple double points (see e.g.\ \cite{GriffithsHarris1994}), 
and it suffices to take its unique Legendrian lifting 
to $X=\P(T^* \CP^2)$ furnished by Proposition \ref{prop:summary} (d). 

Note that $\P(T^*\CP^2)$ is birationally equivalent to $\CP^3$;
however, $\CP^3$ is not the projectivised cotangent bundle of any manifold
(see Boothby \cite{Boothby1962}). 
Recall that $\CP^3$ admits a unique holomorphic contact structure (up to isotopy), induced 
by the $1$-form $\alpha=z_0dz_1-z_1dz_0 + z_2dz_3-z_3dz_2$ on $\C^4$. 
In 1982, Bryant showed that every compact Riemann surface $M$ admits 
a holomorphic Legendrian embedding into $\CP^3$ of the form
\begin{equation}\label{eq:CP3}
	[1: f-\frac{1}{2}g (df/dg) : g : \frac{1}{2}(df/dg)]
\end{equation}
for suitably chosen meromorphic functions $f$ and $g$ on $M$
(see \cite[Theorems F and G]{Bryant1982JDG}). Comparing with the formula
\eqref{eq:liftingcurve} for a Legendrian curve in  $\P(T^*\CP^2)$, which can be written as
%
%
$[1:f:g: -df/dg]$, we see that the curve \eqref{eq:CP3} is obtained from it 
by the quadratic transformation $(x,y,z)\mapsto (x+yz/2,y,-z/2)$ which is an
automorphism on the affine part $\C^3$ of both manifolds.
\qed\end{example}

We wish to point out that the operation of taking the (unique) Legendrian lifting of a nonconstant
holomorphic curve $g\colon M\to Z$ in Proposition \ref{prop:summary} (d) 
may not depend continuously on $g$ at points where the differential  $dg$ vanishes. 
In local coordinates $(z_0,z_1)$ on the surface $Z$ we have $g(t)=(z_0(t),z_1(t))$,
and  the vertical component  equals 
$[-\dot z_1(t):\dot z_0(t)]\in \CP^1$ (see \eqref{eq:liftingcurve}). 
Since the map $\C^2\ni (w_0,w_1)\mapsto [-w_1:w_0]\in\CP^1$
has a point of indeterminacy at $(w_0,w_1)=(0,0)$, this may cause discontinuity of liftings
at branch points of $g$. We have the following more precise observation.

\begin{proposition}\label{prop:singularity}
Assume that $U\subset \C$ is a connected open neighbourhood of $0\in\C$ and
$g^\epsilon = (g^\epsilon_0, g^\epsilon_1):U\to\C^2$ is a family of nonconstant holomorphic
maps, depending continuously on $\epsilon\in\R$ in a neighbourhood of $0\in\R$, 
such that $(g^0)'(0)=(0,0)$. If the meromorphic map
$h^\epsilon=[-(g^\epsilon_1)':(g^\epsilon_0)'] : U\to \CP^1$ 
depends continuously on $\epsilon$,
then for every $\epsilon$ sufficiently close to $0$ there is a point $z=z_\epsilon\in U$ 
close to $0$ such that $(g^\epsilon)'(z_\epsilon)=0$. 
\end{proposition}

It follows that, given a continuous family $g^\epsilon :M\to Z$ of nonconstant holomorphic maps
into a complex surface $Z$ such that $g^\epsilon $ is an immersion for $\epsilon\ne 0$ but $g^0$ 
has a branch point, the family of their Legendrian liftings $f^\epsilon : M\to \P(T^*Z)$
is discontinuous at $\epsilon=0$. 

\begin{proof}
We may assume that $g^0(0)=(0,0)$ and the tangent cone of $g^0$ at $0$ equals $\C\times \{0\}$.
This means that there are integers $2\le k<m$ and numbers $a,b\in\C\setminus \{0\}$ such that
\[	
	g^0_0(z)=a z^k + O(z^{k+1}),\qquad g^0_1(z)=bz^m + O(z^{m+1}).
\]
Hence, 
\[
	h^0(z)=[-mbz^{m-1}+ O(z^{m}): kaz^{k-1}+ O(z^{k})].
\]
Dividing both components by $z^{k-1}$ and letting $z\to 0$ we obtain $h^0(0)=[0:1]\in\CP^1$. 

Pick $r>0$ such that the closed disc $D_r=\{|z|\le r\}$ is contained in $U$ and
$z=0$ is the only zero of the derivative $(g^0_0)'$ on $D_r$. Hence, there is $\epsilon_0>0$ 
such that for every $\epsilon \in [-\epsilon_0,\epsilon_0]$ the function $(g^\epsilon_0)'$ has no zeros 
on $bD_r$. By the argument principle, each of these functions has precisely  
$k-1>0$ zeros on $\mathring D_r$ counted with multiplicities. 
Decreasing $r>0$ and $\epsilon_0>0$ if necessary, the continuity of the map 
$h^\epsilon = [-(g^\epsilon_1)':(g^\epsilon_0)'] : D_r \to \CP^1$ implies that it 
does not assume the value $[1:0]$ for any $\epsilon \in [-\epsilon_0,\epsilon_0]$.
It follows that $(g^\epsilon_1)'$ vanishes at each point $z\in D_r$ where $(g^\epsilon_0)'$
vanishes (to the same or higher order), so the map $g^\epsilon$ has $k-1$ branch points  on $D_r$
counted with multiplicities.
\end{proof}

%
%
%
%
\section{Isotropic Stein submanifolds in projectivised cotangent bundles}\label{sec:Stein}

So far we have mainly focused on Legendrian curves in projectivised cotangent bundles $X=\P(T^*Z)$.
In this section we indicate some results on the existence of isotropic holomorphic maps
from higher dimensional Stein manifolds. (Recall that Stein manifolds of dimension 1 are
precisely the open Riemann surfaces.) 

\begin{proposition}\label{prop:liftingStein}
(Assumptions as above.) 
Let $M$ be a Stein manifold and $Z$ be a complex manifold.
If $\dim Z=\dim M+1$ or $\dim Z > \left[\tfrac 3 2 \dim M\right]$, then every holomorphic immersion 
$g\colon M\to Z$ lifts to a holomorphic isotropic immersion $f\colon M\to X$.
\end{proposition}

\begin{proof}
Write $\dim Z=n+1$. When $n=\dim M$, this is Proposition \ref{prop:summary} (a). 
Assume now that $n\ge \left[\tfrac 3 2\dim M\right]$.
The conormal vector bundle $E \to M$ of $g$, defined in local coordinates by the condition \eqref{eq:lifting},
is then a holomorphic vector bundle of rank $\dim Z-\dim M= n+1-\dim M > \left[\tfrac 1 2 \dim M\right]$.
It follows from the stable rank theorem and the Oka-Grauert
principle that $E$ admits a nowhere vanishing holomorphic section 
(see \cite[Corollary 8.3.9]{Forstneric2017E}). As explained in Sect.\ \ref{sec:basic}, such a section 
determines a holomorphic isotropic lifting $f\colon M\to X$ of $g$.
\end{proof}

%
%

Our next result is a general position theorem for holomorphic isotropic submanifolds
in the projectivised cotangent bundle of an Oka manifold (see \cite[Sect.\ 5.4]{Forstneric2017E}).

\begin{theorem}\label{th:genposOka}
Let $M$ be a Stein manifold and $Z$ be an Oka manifold. 
If $\dim Z\ge 2\dim M$, then every holomorphic map $g_0\colon M\to Z$ can  
be approximated uniformly on compacts in $M$ by holomorphic immersions
$g\colon M\to Z$ which lift to isotropic holomorphic embeddings
$f\colon M\hra X=\P(T^*Z)$.
\end{theorem}

\begin{proof}
By the general position theorem for maps from Stein manifolds to Oka manifolds
(see \cite[Corollary 8.9.3]{Forstneric2017E}), a generic
holomorphic map $g\colon M\to Z$ is an immersion with normal crossings when $\dim Z\ge 2\dim M$,
and is an embedding when $\dim Z> 2\dim M$. By Proposition \ref{prop:liftingStein}, such $g$
admits an isotropic lifting $f\colon M\to X$. The lifted map is an embedding 
since at a normal crossing point $g(x_0)=g(x_1)$ for $x_0\ne x_1 \in M$ the associated complex 
hyperplanes $\Sigma_{x_i}\supset dg_{x_i}(T_{x_i}M)$ for $i=0,1$ (the fibre points of $f$
over $x_0,x_1\in M$) are necessarily distinct. 
\end{proof}

%
%
Next, we give some results on the existence of {\em proper} holomorphic isotropic immersions
of Stein manifolds into projectivised cotangent bundles. 
The notion of (volume) density property has already been
mentioned before (see \cite[Sect.\ 4.10]{Forstneric2017E}).

\begin{corollary}\label{cor:Steinisotropic}
Assume that $Z$ is a Stein manifold of dimension $\ge 2$ with the density or the 
volume density property. Then, every Stein manifold $M$ such that $2\dim M \le \dim Z$ admits 
a proper isotropic holomorphic embedding $M\to X=\P(T^*Z)$.
\end{corollary}

\begin{proof} By \cite{Forstneric2019JAM} (see also \cite{AndristWold2014} for $\dim M=1$),
there exists a proper holomorphic immersion $M\to Z$ with transverse double points, 
and the conclusion follows from Proposition \ref{prop:liftingStein}. 
\end{proof}

\begin{corollary}\label{cor:Steinisotropic2}
Assume that $Z$ is a Stein manifold.  If $M$ is a smoothly bounded strongly pseudoconvex domain
in a Stein manifold such that $2\dim M \le \dim Z$,  then $M$ admits a proper isotropic holomorphic 
immersion $M\to X=\P(T^*Z)$.
\end{corollary}

\begin{proof}
By the main result of \cite{DrinovecForstneric2010AJM}, $M$ admits a proper holomorphic immersion 
$M\to Z$, and the conclusion follows as in the previous corollary.
\end{proof}

%
%

\section{A general position theorem for holomorphic maps}
\label{sec:gen-pos}

Assume that $M$ and $Z$ are connected complex manifolds. 
In this section we prove the following parametric h-principle for the natural inclusion
\begin{equation}\label{eq:nonconstant}
	\Oscr_\mathrm{nc}(M,Z)\longhookrightarrow \Oscr(M,Z)
\end{equation} 
of the space of all nonconstant holomorphic maps  $M\to Z$ into the space of all holomorphic maps.   
This result is important for our subsequent analysis, and certain aspects of its proof will be 
used in the proof of the parametric general position theorem for Legendrian curves given by 
Theorem \ref{th:B}.

\begin{theorem} \label{th:A}
Assume that one of following conditions holds:
\begin{enumerate}[\rm (a)] 
\item $M$ is a compact strongly pseudoconvex domain with Stein interior and $Z$ is a complex manifold. 
\item $M$ is a Stein manifold and $Z$ is an Oka manifold. 
\end{enumerate}
Assume that $Q\subset P$ are compact sets in a Euclidean space and $f \colon M\times P\to Z$ is a 
continuous map satisfying the following two conditions:
\begin{itemize}
\item[\rm (i)] $f_p=f(\cdotp,p)\colon M\to Z$ is a holomorphic map for every $p\in P$, and
\item[\rm (ii)] $f_p\in \Oscr_{\mathrm{nc}}(M,Z)$ is a nonconstant holomorphic map for every $p\in Q$.
\end{itemize}
Then there exists a homotopy $f^t \colon M\times P\to Z$ $(t\in [0,1])$ such that 
$f^t_p := f^t(\cdotp,p)\in \Oscr(M,Z)$ for every $(p,t)\in P\times [0,1]$ and the following conditions hold:
\begin{itemize}
\item[\rm (1)]  $f^t_p =f_p$ \ for every $(p,t)\in (P\times \{0\}) \cup (Q\times [0,1])$, and 
\item[\rm (2)]  
$f^1_p\in \Oscr_\mathrm{nc}(M,Z)$  is nonconstant for every $p\in P$.
\end{itemize}
\end{theorem}

Under the assumption (a) of the theorem, a holomorphic map from $M$ may either mean a
map holomorphic on an open neighbourhood of $M$ in the ambient complex manifold,
or a map of class $\Ascr^r(M,Z)$ for some $r\in \{1,2,\ldots,\infty\}$; the result holds for all classes.
To simply the notation we shall write $\Oscr(M,Z)$ in both cases. When speaking of continuous 
families of functions or maps holomorphic on neighbourhoods of $M$, it is always 
assumed that all maps in the family are holomorphic on the same neighbourhood which 
however may shrink in the course of a proof. (See Sect. \ref{sec:h-principle} for a
justification of this point.)

The next corollary follows in a standard way from Theorem \ref{th:A};
see e.g.\ \cite[Proof of Corollary 5.5.6]{Forstneric2017E}.

%
%
\begin{corollary}  \label{cor:gen-pos-A}
If $M$ and $Z$ satisfy one of the assumptions {\rm (a), (b)} in Theorem \ref{th:A}, then the inclusion 
\eqref{eq:nonconstant} of the space of all nonconstant holomorphic maps $M\to Z$ into the space of all
holomorphic maps is a weak homotopy equivalence. 
\end{corollary}

\begin{proof}[Proof of Theorem \ref{th:A}]
By the hypothesis we have that  $Q\subset P\subset \R^m$ for some $m\in\N$. 

We begin by considering case (a) with $Z=\C$.
Since the function $f_p\in\Oscr(M)$ is assumed to be
nonconstant for every $p$  in the compact set $Q$, 
there are finitely many pairwise distinct points $x_1,\ldots, x_n\in M$  such that 
$f_p$ is nonconstant on the set $\{x_1,\ldots, x_n\}$ for every $p\in Q$. 
Clearly we may assume that $2n-2 >m$. Consider the map $F=(F_1,\ldots,F_n)\colon P \to \C^n$ defined by
\[ 
	F(p) = \bigl(f_p(x_1),\ldots, f_p(x_n)\bigr) \in \C^n,\qquad p\in P.
\] 
Let 
\begin{equation}\label{eq:Delta}
	\Delta=\{(z_1,\ldots,z_n) \in \C^n: z_1=z_2=\ldots = z_n\}. 
\end{equation}
For a fixed $p\in P$ the map $f_p\colon M\to Z$ is nonconstant on the set
$\{x_1,\ldots, x_n\}\subset M$ if and only if $F(p)\notin \Delta$.
In particular, we have that $F(p)\notin \Delta$ for all $p\in Q$; equivalently, 
$F(Q)\cap \Delta=\varnothing$.  Using Tietze's theorem we may extend $F$ 
from $P$ to a continuous map $F:\R^m\to \C^n$, and by Weierstrass's 
theorem we can approximate $F$ as closely as desired uniformly on a neighbourhood of $P$
by a polynomial map $\wt F\colon \R^m\to\C^n$. 
The diagonal $\Delta\subset \C^n$ \eqref{eq:Delta} is a complex line of real codimension $2n-2$. 
Since $2n-2 >m$, the transversality theorem shows that a generic smooth
perturbation of $\wt F$ satisfies $\wt F(P)\cap \Delta=\varnothing$. 
There is a neighbourhood $Q_0\subset P$ 
of $Q$ such that $F(Q_0)\cap \Delta=\varnothing$. Assuming as we may that 
$\wt F$ is close enough to $F$ on $Q_0$, every map $F_t=(1-t)F+t\wt F$ $(t\in[0,1])$ 
satisfies $F_t(Q_0)\cap \Delta=\varnothing$. 

Choose a continuous function $\chi\colon P \to [0,1]$ supported on $Q_0$
which equals $1$ on $Q$. Consider the map 
\[
	G=(G_1,\ldots, G_n) = \chi F+(1-\chi)\wt F\colon P \to \C^n.
\]
Clearly, $G=F$ on $Q$, $G=\wt F$ on $P\setminus Q_0$, and $G(p)$ is a convex combination
of $F(p)$ and $\wt F(p)$ for every $p\in Q_0\setminus Q$. By the choice of $\wt F$ 
it follows that $G(P)\cap \Delta=\varnothing$.
 
Let $\phi_i \in \Oscr(M)$ for $i=1,\ldots, n$ be chosen such that 
$\phi_i(x_j)=\delta_{i,j}$ (Kronecker's delta) for $i,j=1,\ldots, n$. Consider the 
continuous function $\tilde f \colon M\times P\to \C$ defined by
\[
	\tilde f(x,p)= \tilde f_p(x) := f_p(x) + \sum_{j=1}^n (G_j(p)-F_j(p)) \phi_j(x),\qquad x\in M,\ p\in P.
\]
Note that $\tilde f_p\in \Oscr(M)$ for every $p\in P$,  $\tilde f_p=f_p$ for $p\in Q$, and 
\[
	\tilde f_p(x_i)=f_p(x_i)+G_i(p)-F_i(p) = G_i(p),\quad  i=1,\ldots,n,\ p\in P.
\]
Since $G(P)\cap \Delta=\varnothing$, $\tilde f_p$ is nonconstant on $\{x_1,\ldots, x_n\}$ for every 
$p\in P$. Note that $\tilde f$ approximates $f$ as closely as desired uniformly
on $M\times P$ provided $\wt F$ is close enough to $F$ on $P$. 
A homotopy from $f_p$ to $\tilde f_p$ is obtained by taking 
$(1-t)f_p+t\tilde f_p$ for $t\in [0,1]$.

This establishes case (a) of the theorem when $Z=\C$. 
An obvious modification of the proof applies when $Z=\C^k$ for some $k\in\N$.
If $M$ is a Stein manifold and $Z=\C^k$, we apply the already proven result
on any nontrivial compact strongly pseudoconvex and $\Oscr(M)$-convex domain $M'\subset M$;
the conclusion then follows from the parametric Oka-Weil theorem (see \cite[Theorem 2.4.8]{Forstneric2017E}).

Assume now that $Z$ is an arbitrary complex manifold (in case (a)), resp.\ an Oka manifold (in case (b)). 
Consider first the case when $M$ is a strongly pseudoconvex domain.  Set 
\[
	P_0=\{p\in P: f_p\ \ \text{is constant}\}.
\]
If $p\in P_0$ and $q\in P$ is close to $p$, then $f_q(M)$ is contained in a coordinate
neighbourhood of the point $f_p(M)$ in $Z$. Hence, by compactness of $Q$ 
there are finitely many pairs of open sets $V'_j\Subset V_j\subset P$ 
and coordinate neighbourhoods $U_j\subset Z$ for $j=1,\ldots,r$ such that 
\[
	f_p(M)\subset U_j\ \ (p\in \overline V_j), \qquad P_0 \subset \bigcup_{j=1}^r V'_j. 
\]
This means that for all $p\in \overline V_j$ we may consider $f_p$ as a map with values in 
$\C^k$ with $k=\dim Z$. Choose a compact set $Q_0\subset P$ such that
\[
	Q\subset \mathring Q_0,\qquad P_0\cap Q_0=\varnothing,\qquad 
	P\setminus Q_0 \subset \bigcup_{j=1}^r V'_j.
\]

The desired homotopy $f^t_p$ will be constructed in $r$ steps. At the $j$-th step we shall deform
the given family of maps from the $(j-1)$-st step for the parameter values $p\in V_j$, 
and the homotopy will be fixed for $p\in P\setminus V_j$. 

In the first step, we apply the already explained argument in the special case  $Z=\C^k$
to the initial family of maps $f_p\colon M\to U_1$ $(p\in \overline V_1)$ into the coordinate neighbourhood
$U_1\subset Z$, approximating it as closely as desired by a continuous family of 
nonconstant holomorphic maps  $\tilde f_p\colon M\to U_1$ $(p\in \overline V_1)$. 

We shall now modify the family $\{\tilde f_p\}_{p\in \overline V_1}$ in order to glue it with the initial
family $\{f_p\}_{p\in P}$ for the parameter values $p\in P\setminus V_1$. Choose an open set $V''_1$ in $P$, 
with $V'_1\Subset V''_1\Subset V_1$, and a continuous function $\chi\colon P\to[0,1]$ such that 
\[
	\chi=1\ \text{on}\ V'_1\setminus Q_0 \quad \text{and}
	\quad \chi=0\  \text{on}\ (V_1\cap Q) \cup (V_1\setminus V''_1).
\]
Define a new continuous family  $f^1_p\in\Oscr(M)$ $(p\in V_1)$ by
\[
	f^1_p=\chi(p) \tilde f_p + (1-\chi(p)) f_p.
\]
Since $\chi=0$ on $V_1\setminus V''_1$, the family $f^1_p$ extends 
continously to all $p\in P$ by setting $f^1_p=f_p$ for $p\in P\setminus V_1$. 
For $p\in V'_1\setminus Q_0$ we have that $\chi(p)=1$ and hence $f^1_p=\tilde f_p$ 
which is nonconstant on $\{x_1,\ldots, x_n\} \subset M$ 
by the choice of $\tilde f_p$. For $p\in V_1\cap Q_0$ the map $f'_p$ is nonconstant since $f_p$ is 
nonconstant and $\tilde f_p$ is chosen close enough to $f_p$ so that any convex combination of the
two maps (in particular, $f^1_p$) is nonconstant. Thus, $f^1_p$ is nonconstant for all 
$p\in Q_1:=Q_0\cup \overline{V'_1}$. It follows that 
\[
	P_1:=\left\{ p\in P : f^1_p\ \text{is constant} \right\} 
	\subset P\setminus Q_1 \subset \bigcup_{j=2}^r V'_j. 
\]
Assuming as we may that $\tilde f_p$ approximates $f_p$ sufficiently closely 
uniformly on $M$ for all $p\in \overline{V_1}$,
it follows that $f^1_p(M)\subset U_j$ for all $p\in \overline {V_j}$, $j=2,\ldots, r$.
Hence, we can apply  the same argument as in the first step to the family $f^1_p$ 
on the parameter set $p\in V_2$, thereby getting a new family $f^2_p\in \Oscr(M)$ $(p\in P)$
such that $f^2_p$ is nonconstant for all 
$p\in Q_2:=Q_1\cup\overline {V'_2}=Q_0 \cup \overline {V'_1} \cup \overline {V'_2}$ 
and $f^2_p=f^1_p$ for all $p\in P\setminus V_2$. 
After $r$ steps of this kind we obtain a family $f^r_p\in \Oscrnc(M)$ $(p\in P)$ 
of nonconstant holomorphic maps such that $f^r_p=f_p$ for all $p\in Q_0$, as well as a homotopy from $f_p$ to $f^r_p$. Renaming $f^r_p$ to $f^1_p$ we thus obtain the desired conclusion.

Finally, if $M$ is a Stein manifold and $Z$ is an Oka manifold, we first apply the above construction
on any compact strongly pseudoconvex and $\Oscr(M)$-convex domain $M'\subset M$;
the conclusion then follows from the parametric Oka principle, with approximation on $M'$,
for holomorphic maps $M\to Z$ with respect to the pair of compact parameter spaces $Q\subset P$
(see \cite[Theorem 5.4.4]{Forstneric2017E}). 
\end{proof}

%
%

\section{Approximating vertical Legendrian curves by horizontal ones}
\label{sec:approx}

Let $Z$ be a connected complex manifold and $X=\P(T^*Z)$ be its projectivised 
cotangent bundle, endowed with the standard complex contact structure (see Sect.\ \ref{sec:intro}). 
Let $\pi\colon X\to Z$ denote the base projection. 
Given a connected Riemann surface $M$, we denote by $\Lscr(M,X)$ the space of all 
Legendrian holomorphic maps $M\to X$, endowed with the compact-open topology. 
We have seen in the introduction that the postcomposition map
\[ 
	\pi_* : \Lscr(M,X) \longrightarrow \Oscr(M,Z), \qquad f \longmapsto \pi\circ f, 
\]
has two kinds of fibres.  The fibre over a constant map $M\to \{z^0\}\subset Z$ consists of 
all holomorphic maps into the fibre $\pi^{-1}(z^0)$; we call such maps \emph{vertical}.  
Clearly, every vertical map is Legendrian. Nonvertical Legendrian maps $f\in \Lscr(M,X)$, 
i.e., those with nonconstant projection $g=\pi\circ f\colon M\to Z$, are called \emph{horizontal}.

In this section we prove several results concerning the existence of approximations
and deformations of vertical Legendrian maps to horizontal ones. 
In the introduction we have indicated some reasons
why such approximations do not exist in general. 
On the other hand, we will now show that approximation by horizontal Legendrian curves is 
always possible if $M$ is a compact bordered Riemann surface (with nonempty boundary). 
It is classical (see e.g.\ Stout \cite{Stout1965}) 
that such a surface is conformally equivalent to a compact domain 
with real analytic boundary in an open Riemann surface $\wt M$.
We denote by $\Oscr(M, X)$ the space of germs of holomorphic maps from open neighbourhoods of $M$ 
in $\widetilde M$ into $X$, endowed with the colimit topology; see Sect.\ \ref{sec:h-principle} for more details.
When speaking of homotopies of maps in $\Oscr(M,X)$,
we assume that all maps in the family are defined on the same neighbourhood of $M$;
this is justified by Proposition \ref{p:map-from-compact}.
Our arguments and results also apply to Legendrian curves of  class 
$\Ascr^r(M,X)=\Cscr^r(M,X)\cap\Oscr(\mathring M,X)$ for any integer $r\ge 1$.
We shall denote the space of Legendrian curves $M\to X$ 
in any of these classes by $\Lscr(M,X)$; it will be clear from 
the context which class is meant.

We begin with the following basic result on approximation of vertical holomorphic 
curves by horizontal holomorphic Legendrian curves.

%
%
\begin{proposition} \label{prop:approx1}
Let $Z$ be a connected complex manifold of dimension at least $2$ and let $X=\P(T^*Z)$.
Assume that $M$ is a compact bordered Riemann surface. 
Then, every vertical holomorphic map $f\in \Oscr(M,X)$ into a fibre $X_{z^0}$ can be 
approximated uniformly on $M$ by horizontal Legendrian maps $f_1\in \Lscr(M,X)$.
Furthermore, $f_1$ can be chosen such that there is a homotopy $f_t\in \Lscr(M, X)$ $(t\in [0,1])$ 
satisfying the following conditions:
\begin{enumerate}[\rm (i)]
\item $f_0=f$, 
\item the map $f_t$ is horizontal and uniformly close to $f$ for every $t\in (0,1]$, and 
\item every map $f_t$ in the family has the same vertical component as $f=f_0$
with respect to a chosen trivialisation of the restricted bundle $X|_U\to U$ over 
a contractible open neighbourhood $U\subset Z$ of $z^0$ containing $\bigcup_{t\in [0,1]} f_t(M)$.
\end{enumerate}
\end{proposition}

\begin{proof}
By the hypothesis, the image $f(M)$ lies in the fibre of the projection $\pi\colon X\to Z$
over some point $z^0\in Z$. Let $\dim Z=n+1\ge 2$. 
Choose holomorphic coordinates $z=(z_0,\ldots,z_n)$ on 
a neighbourhood $U\subset Z$ of $z^0$ such that $z(U)\subset \C^{n+1}$ is a ball
centred at $z(z^0)=0\in\C^{n+1}$. 
Let $[\zeta_0:\cdots:\zeta_n]$ be the corresponding homogeneous coordinates
on $\P(T_{z}^* Z)\cong \CP^n$ for $z\in U$. In these coordinates, the contact structure
on $X_U =\P(T^* U)$ equals $\xi=\ker\eta$ where $\eta= \sum_{i=0}^n \zeta_i \, dz_i$
(see \eqref{eq:eta}), and $f=(0,\tilde h)$ for some holomorphic map 
$\tilde h\colon M\to\CP^n$.  Let $\tau\colon \C^{n+1}_*\to \CP^n$ be the tautological projection.
Since $M$ has the homotopy type of a bouquet of circles, $\tilde h$ lifts 
to a map $h=(h_0,\ldots, h_n) \colon M\to \C^{n+1}_*$ which may be chosen holomorphic by the
Oka principle (see \cite[Lemma 5.1]{AlarconForstnericLopez2019JGEA}).
To prove the proposition, we shall find a nonconstant holomorphic map 
$\tilde g=(\tilde g_0,\ldots,\tilde g_n)\colon M\to z(U)\subset \C^{n+1}$ such that 
\begin{equation}\label{eq:main}
	\sum_{i=0}^n h_i \, d\tilde g_i =0\quad \text{on}\ M.
\end{equation}
The holomorphic map $f_1=(z^{-1}\circ \tilde g,\tilde h):M\to X$ is then horizontal Legendrian, 
and it approximates $f$ provided $\tilde g$ is small. 
Furthermore, $f_t=(z^{-1}\circ (t\tilde g),\tilde h)$ for $t\in [0,1]$ is a homotopy 
of holomorphic Legendrian maps from $f_0=f$ to $f_1$ satisfying the proposition.

Note that at least one of the component functions $h_i$ of $h$ is not identically zero;
assume that this holds for $h_n$. The equation \eqref{eq:main} is then equivalent to
\begin{equation}\label{eq:main2}
	d\tilde g_n = - \sum_{i=0}^{n-1} \frac{h_i}{h_n} \, d\tilde g_i =: \beta.
\end{equation}
We shall find functions $\tilde g_0,\ldots,\tilde g_{n-1}\in\Oscr(M)$ with arbitrarily small norm, 
not all constant, such that the $1$-form $\beta$ is holomorphic and exact on $M$,  i.e., it 
has vanishing periods on a basis of the homology group $H_1(M;\Z)$. 
A function $\tilde g_n$ satisfying \eqref{eq:main2} is then obtained by integration,
thereby providing a nonconstant solution to \eqref{eq:main}.

Note that the map $\tilde h\colon M\to \CP^n$ is constant if and only if all
quotients $h_i/h_n$ for $i=0,\ldots, n-1$ are constant on $M$; 
in this case the $1$-form $\beta$ \eqref{eq:main2}
is exact for every choice of functions $\tilde g_0,\ldots,\tilde g_{n-1}\in\Oscr(M)$. 

Suppose now that $\tilde h$ is not constant, 
and assume without loss of generality that $h_0/h_n$ is nonconstant.
In order to ensure that the $1$-form 
$\beta$ \eqref{eq:main2} is holomorphic on $M$, we seek functions $\tilde g_i$ 
$(i=0,\ldots,n-1)$ of the form $\tilde g_i=h_n^2 g_i$ with $g_i\in\Oscr(M)$.
Then, $d\tilde g_i=h_n^2 dg_i + 2h_n g_i dh_n$ and hence 
\begin{equation}\label{eq:beta}
	\beta = - \sum_{i=0}^{n-1} h_i(h_n dg_i + 2g_i dh_n).
\end{equation}
This reduction is unnecessary if the function $h_n$ is nowhere vanishing on $M$.

When $g_i=0$ for all $i=0,\ldots,n-1$, we have $\beta =0$ which is exact.
We will now show that for an arbitrary choice of functions $g_1,\ldots,g_{n-1}\in \Oscr(M)$
with sufficiently small norm, the function $g_0=g_0(g_1,\ldots,g_{n-1})\in\Oscr(M)$ 
can be chosen such that $\beta$ is exact.

Let $C_1,\ldots, C_l\subset M$ be smooth closed curves forming
a basis of the homology group $H_1(M;\Z)\cong \Z^l$ and such that
the compact set $C=\bigcup_{j=1}^l C_j$ is $\Oscr(M)$-convex. The condition 
that $\beta$ have vanishing periods is then 
\[
	\int_{C_j} h_0h_n dg_0 + 2g_0 h_0 dh_n = \alpha_j\in \C,\qquad j=1,\ldots,l,
\]
where the constants $\alpha_j\in \C$ are obtained by integrating the remaining terms 
with indices $i=1,\ldots, n-1$ in $\beta$. (In particular, these constants 
depend on  the functions $g_1,\ldots,g_{n-1}$.) Furthermore, integrating the first term 
$h_0h_n dg_0$ by parts we obtain the condition
\begin{equation}\label{eq:omega}
	\int_{C_j} g_0 \, \omega = \wt \alpha_j,\qquad j=1,\ldots, l,
\end{equation}
for some other numbers $\wt \alpha_j\in \C$ depending on the functions  $g_1,\ldots,g_{n-1}\in \Oscr(M)$, 
where
\begin{equation}\label{eq:omega1}
	\omega= 2h_0 dh_n - d(h_0h_n) = h_0 dh_n - h_n dh_0 = h_n^2\, d(h_0/h_n). 
\end{equation}
Since $h_0/h_n$ is nonconstant by the assumption, the holomorphic $1$-form 
$\omega$ does not vanish identically on $M$, and hence it does not vanish identically 
on any nontrivial arc in $M$ in view of the identity principle. 
This allows us to find functions $\xi_1,\ldots,\xi_l \in \Oscr(M)$ such that 
\begin{equation}\label{eq:omega2}
	\int_{C_j} \xi_k \, \omega = \delta_{j,k}, \qquad j,k=1,\ldots, l.
\end{equation}
(As before, $\delta_{i,j}$ denotes Kronecker's delta.) In addition, the $\xi_k$'s may be chosen to vanish at
a given pair of distinct points $x^0, x^1\in M$. To find such functions, we first construct smooth functions 
$\xi_k \colon C=\bigcup_{j=1}^l C_j \to\C$ satisfying \eqref{eq:omega2} and vanishing at $x^0, x^1\in M$.
This can easily be done using Gromov's {\em convex integration lemma}
\cite[Lemma 2.1.7]{Gromov1973IZV}; see also Spring \cite[Theorem 3.4]{Spring2010}
for a parametric version. (Similar results have been proved in \cite[Lemma 7.3]{AlarconForstneric2014IM} 
and in \cite[Lemma 2.3]{AlarconForstnericLopez2019JGEA}; see also 
\cite[Lemma 3.8]{AlarconForstneric2019JAMS}.) 
Since $C$ is $\Oscr(M)$-convex, we can apply the Mergelyan approximation theorem to find
holomorphic functions $\tilde \xi_k\in \Oscr(M)$ $(k=1,\ldots, l)$ vanishing at $x^0$ and $x^1$
such that the period $l\times l$ matrix $A=\left(\int_{C_j} \tilde \xi_k \, \omega\right)=(a_{j,k})$ is very close 
to the identity. Finally, taking $\xi_k= \sum_{i=1}^l a^{i,k}\tilde \xi_i$ $(k=1,\ldots,l)$ with $A^{-1}=(a^{i,k})$ gives
functions satisfying condition \eqref{eq:omega2}.

It follows from \eqref{eq:omega2} that for every collection of functions $\xi_0,g_1,\ldots,g_{n-1}\in \Oscr(M)$ 
there is a unique $t=(t_1,\ldots,t_l)\in\C^l$ such that the function
\begin{equation}\label{eq:g0}
	g_0:=\xi_0 +\sum_{k=1}^l t_k \xi_k \in \Oscr(M) 
\end{equation}
satisfies the period conditions \eqref{eq:omega}. By choosing $\xi_0,g_1,\ldots,g_{n-1}\in \Oscr(M)$ 
close to the zero function, $t$ is also close to $0\in \C^l$ and hence $g_0$ is as small as desired. 
Furthermore, choosing $\xi_0$ such that $\xi_0(x^0)\ne \xi_0(x^1)$ 
we get that  $g_0(x^0)\ne g_0(x^1)$, so $g_0$ is nonconstant. 
\end{proof}

%
%
\begin{remark}\label{rem:approx1}
The proof of Proposition \ref{prop:approx1} generalizes to the following parametric case. 
Let $P$ be a compact Hausdorff space and $\{f_p:p\in P\}\subset  \Lscr(M,X)$ be a continuous family of 
holomorphic Legendrian curves whose base projections 
$g_p=\pi\circ f_p\in \Oscr(M,Z)$ $(p\in P)$ take values in a contractible coordinate chart $U$ of $Z$. Assume that the vertical
components $\tilde h_p\colon M\to\CP^n$ of $f_p$ admit a continuous family of liftings 
$h_p= (h_{p,0},\ldots, h_{p,n}) \colon M\to \C^{n+1}_*$ for $p\in P$. 
(By \cite[Lemma 5.1]{AlarconForstnericLopez2019JGEA} this holds
if $P$ is contractible.) Assume  that one of the components of $h_p$, 
say $h_{p,n}$, is not identically zero and one of the quotients $h_{p,i}/h_{p,n}$ 
is nonconstant for every $p\in P$. (The indices $i$ and $n$ should not depend on $p\in P$.)
Then, there is a homotopy $f^t_p\in \Lscr(M,X_U)$ $(p\in P,\ t\in [0,1])$ 
such that for each $p\in P$ we have that
\begin{enumerate}[\rm (i)]
\item $f^0_p=f_p$, 
\item the map $f^t_p$ is horizontal and uniformly close to $f_p$ for every $t\in (0,1]$, and 
\item every map $f^t_p$ for $t\in [0,1]$ has the same vertical component as $f_p=f^0_p$
(compare with Proposition \ref{prop:approx1} (iii)).
\end{enumerate}
To prove this result, we apply the proof of Proposition \ref{prop:approx1} with continuous 
dependence of all objects on the parameter $p\in P$. In particular, we find finitely many continuous families 
of functions $\xi_{p,k}\in \Oscr(M)$ $(p\in P)$ for $k=1,\ldots, N$ such that 
the vectors
\[
	\biggl( \int_{C_j} \xi_{p,k} \omega_p \biggr)_{\! j=1,\ldots,l} \in \C^l,\qquad k=1,\ldots,N, 
\] 
span $\C^l$ for every $p\in P$. (Here, $\omega_p$ is the $1$-form \eqref{eq:omega1}
depending on $p\in P$.) The concluding argument is the same as before, except 
that we now use the implicit function theorem to find numbers $t_1,\ldots,t_N\in \C$,
depending continuously on $p\in P$,  which ensure that the functions
$g_{p,0}$ defined by \eqref{eq:g0} satisfy the period conditions \eqref{eq:omega}
for all $p\in P$. Note that families of nonconstant perturbations
of the component functions $g_{p,j}$, depending continuously on $p\in P$, exist by Theorem \ref{th:A}.
\qed\end{remark}

Proposition \ref{prop:approx1} implies the following result for an open bordered Riemann surface $M$ (the interior of a compact bordered Riemann surface $\overline M$).

\begin{corollary}\label{cor:approx1}
Assume that $M$ is an open bordered Riemann surface. Then, every vertical holomorphic
map $f\in \Oscr(M, X)$ can be approximated uniformly on compacts in $M$ by horizontal
holomorphic Legendrian maps $f_1\in \Lscr(M,X)$. If in addition 
$f\in \Ascr^k(M,X) = \Cscr^k(\overline M,X)\cap \Oscr(M,X)$, $k\geq 1$, then the approximation is possible in $\Cscr^k(\overline M,X)$.
\end{corollary}

\begin{proof}
We may assume that $M$ is a domain with smooth  boundary in a Riemann surface $\wt M$.
Let $K$ be a compact $\Oscr(M)$-convex set in $M$. We first approximate $f$ uniformly
on $K$ by a holomorphic vertical map $\tilde f\in \Oscr(\overline M,X)$; this is possible by the
Oka principle since the fibres of $X\to Z$ are projective spaces $\P(T^*_z Z)\cong \CP^n$
(see \cite[Theorem 5.4.4]{Forstneric2017E}). It remains to apply Proposition \ref{prop:approx1}
to the map $\tilde f$. If $f\in \Ascr^k(M,X)$, then we approximate $f$ in 
$\Cscr^k(M,X)$ by a holomorphic vertical map $\tilde f\in \Oscr(\overline M,X)$ and proceed as before \cite[Theorem 8.11.4]{Forstneric2017E}.
\end{proof}

Assuming that the base $Z$ is an Oka manifold, 
we now give a couple of global approximation results for Legendrian curves in $\P(T^*Z)$
parameterised by an arbitrary open Riemann surface. The first one pertains to vertical maps.

%
%
\begin{corollary}\label{cor:approx2}
If $M$ is an open Riemann surface and $Z$ is an Oka manifold, then every vertical holomorphic map 
$f:M\to X=\P(T^*Z)$ can be approximated uniformly on compacts in $M$ 
by horizontal Legendrian maps $f_1:M\to X$. If in addition $Z$ is a 
Stein manifold with the density or the volume density property, then $f_1$ can be chosen proper.
\end{corollary}

Recall that a complex manifold $Z$ is said to enjoy the {\em density property} if the Lie algebra generated 
by all $\C$-complete holomorphic vector fields is dense in the Lie algebra of all holomorphic vector fields on 
$Z$ (see Varolin \cite{Varolin2001,Varolin2000} or \cite[Sect.\ 4.10]{Forstneric2017E}; in the latter source 
and in \cite{KalimanKutzschebauch2015} the reader can find a list of known examples of such manifolds).  
Similarly one defines the {\em volume density property} of a complex manifold endowed with a holomorphic 
volume form, as well as the algebraic versions of these properties (see Kaliman and Kutzschebauch 
\cite{KalimanKutzschebauch2010IM}).  Every Stein manifold with the density property is an Oka manifold 
(see \cite[Proposition 5.6.23]{Forstneric2017E}).

\begin{proof}
Assume that $f(M)\subset X_{z^0}$ for some $z^0\in Z$; 
then $f=(z^0,h)$ with $h\colon M\to X_{z^0}\cong\CP^n$.
Let $V\subset Z$ be a contractible coordinate neighbourhood  of the point $z^0$.
Choose a compact $\Oscr(M)$-convex set $K\subset M$.
By Proposition \ref{prop:approx1} we can find a neighbourhood $U\subset M$ of $K$ and
a holomorphic Legendrian map $\tilde f = (\tilde g,h) \colon U\to X_V$ such that $\tilde g\colon U\to V$ is 
nonconstant and the vertical component of $\tilde f$ agrees with the vertical component of $f$
(with respect to the coordinates on $V$).
Let $x_1,\ldots, x_m\in K$ be the finitely many points of $K$ at which the differential $d\tilde g$ vanishes. 
Thus, there is a neighbourhood $U'\subset U$ of $K$ such that $\tilde g$ is an immersion on 
$U'\setminus  \{x_1,\ldots, x_m\}$. Since $V$ is contractible, there is a 
continuous map $M\to V$ which agrees with $\tilde g$ in some neighbourhood of $K$.
Since $Z$ is an Oka manifold, there is a 
holomorphic map $g_1:M\to Z$ which approximates $\tilde g$ as closely as desired on 
a neighbourhood of $K$ and agrees with $\tilde g$ to any given order $k_j\in\N$ 
at the point $x_j$ for each $j=1,\ldots,m$. (See \cite[Theorem 5.4.4]{Forstneric2017E}.)
Let $f_1\colon M\to X$ be a lifting of $g_1$ furnished by Proposition \ref{prop:summary}.
If the integer $k_j$ is chosen bigger than the order of vanishing of the differential
$d\tilde g$ at the point $x_j$ for every $j=1,\ldots,m$, 
then the tangent plane to $g_1$ at $x_j$ (a complex line in $T_{g_1(x_j)}Z$)  
agrees with the tangent plane of $\tilde g$ at $x_j$ for every $j=1,\ldots,m$. Hence, the 
Legendrian lifting $f_1$ of $g_1$ can be chosen to approximate $\tilde f$ (and hence $f$)
uniformly on $K$ and to agree with $\tilde f$ at each of the points $x_1,\ldots, x_m$.
We can also arrange that $f_1$ is an embedding
on $M\setminus U$ where $U$ is a neighbourhood of $K$. Indeed, the map $g_1\colon M\to Z$
can be chosen an immersion with simple double points (an embedding if $\dim Z\ge 3$)
on $M\setminus K$; see \cite[Corollary 8.9.3]{Forstneric2017E}. 
By Proposition \ref{prop:summary} (b), such a map $g_1$ admits a Legendrian lifting $f_1$ 
which is an embedding on $M\setminus K$. 

If $Z$ is a Stein manifold with the density or the volume density property, 
then the map $g_1:M\to Z$ as above can be chosen proper (see Andrist and Wold \cite{AndristWold2014} 
or \cite[Theorem 9.8.6]{Forstneric2017E}). Hence, any Legendrian lifting 
$f_1\colon M\to X$ of $g_1$ is also proper.
\end{proof}

\begin{problem}
In the context of Corollary \ref{cor:approx2}, is it possible to connect 
a given vertical holomorphic map $f\colon M\to X$ to a horizontal Legendrian 
map $f_1:M\to X$ by a homotopy of Legendrian maps $M\to X$ 
as in Proposition \ref{prop:approx1}?
\end{problem}

The proof of Corollary \ref{cor:approx2} also yields the following result concerning 
the approximation of horizontal Legendrian curves when the base manifold $Z$ is Oka
or a Stein manifold with the density property. 
(We refer to \cite[Sects.\ 4.10 and 5.4]{Forstneric2017E} for these notions.)

%
%
\begin{corollary}\label{cor:approx3}
Assume that $M$ is an open Riemann surface and $Z$ is an Oka manifold.
Let $K$ be a compact $\Oscr(M)$-convex subset of $M$, and let $f\colon U\to X=\P(T^*Z)$
be a horizontal holomorphic Legendrian map from an open neighbourhood $U\subset M$ of $K$.
Assume that the map $g=\pi\circ f\colon U\to Z$ extends from a smaller neighbourhood of $K$ to a 
continuous map $M\to Z$. Then, $f$ can be approximated uniformly on $K$ by horizontal holomorphic 
Legendrian maps $\tilde f\colon M\to X$. If in addition $Z$ is a Stein manifold with the density property,
then $\tilde f$ can be chosen proper and an embedding on $M\setminus K$.
\end{corollary}

\begin{proof}
The conditions imply that the projection $g=\pi\circ f\colon U\to Z$ can be approximated 
uniformly on a neighbourhood of $K$ by a globally defined holomorphic map $\tilde g\colon M\to Z$
which agrees with $g$ to an arbitrary finite order at each of the finitely many branch points
of $g$ on $K$. If in addition $Z$ is a Stein manifold with the density property,
then such $g$ can be chosen proper. 
The rest of the argument is exactly as in the proof of Corollary \ref{cor:approx2}.
\end{proof}

\begin{problem}
Let $Z$ be a complex manifold of dimension $n+1\ge 3$ and $M$ be a compact strongly
pseudoconvex domain in a Stein manifold $S$ with $\dim S\le n$.
Assume that $f\colon M\to X=\P(T^*Z)$ is a vertical holomorphic map.
Is it possible to approximate $f$ uniformly on $M$ by horizontal isotropic
holomorphic maps $M\to X$?
\end{problem}

%
%

\section{A parametric general position theorem for Legendrian curves in $\P(T^*Z)$} 
\label{sec:HPLegendrian}

In this section we prove a parametric version of Proposition \ref{prop:approx1}
provided that all vertical Legendrian maps in the given family are nondegenerate in the following sense.

%
%
\begin{definition}\label{def:nondegenerate}
Let $M$ be a Riemann surface. A vertical holomorphic map $f\colon M\to X=\P(T^*Z)$
with $f(M)\subset X_{z^0}=\P(T^*_{z^0} Z)$ for some point $z^0\in Z$ is {\em degenerate} 
if the image $f(M)$ is contained in a proper projective linear subspace of $\P(T^*_{z^0}Z)\cong \CP^{n}$; 
otherwise $f$ is {\em nondegenerate}.  In particular, if $\dim Z=2$ then a vertical map $M\to \P(T^*M)$ 
is degenerate if and only if it is constant.
\end{definition}

Let $z=(z_0,\ldots, z_n)$ be holomorphic coordinates on a neighbourhood $U\subset Z$
of the point $z^0\in Z$. A vertical holomorphic curve $f\colon M\to X$ with image in 
the fibre $X_{z^0}$ is of the form $f(x)=(z^0,[h_0(x):\cdots : h_n(x)])$ for some holomorphic map
$h=(h_0,\ldots, h_n)\colon M\to \C^{n+1}_*$, and $f$ is 
degenerate if and only if the image of $h$ is contained in a proper complex subspace of $\C^{n+1}$. 
The set of all degenerate holomorphic maps $M\to \CP^n$ 
from an open or compact bordered Riemann surface $M$ is closed and nowhere dense 
in the space of all holomorphic maps $\Oscr(M,\CP^n)$. 


In the sequel we assume that $M$ is a compact bordered Riemann surface, and we 
use the same notation $\Oscr(M,X)$ and conventions as in Sect.~\ref{sec:approx}.

%
%
\begin{theorem} \label{th:B}
Assume that $Z$ is a complex manifold, $X=\P(T^*Z)$ is the complexified cotangent bundle of $Z$ with 
its canonical complex contact structure, and $Q\subset P$ are compact sets in a Euclidean space $\R^m$
such that every point $p_0\in P$ has a basis of connected contractible neighbourhoods.  
Let $M$ be a compact bordered Riemann surface.
Assume that $f \colon M\times P\to X$ is a continuous map satisfying the following conditions:
\begin{enumerate}[\rm (a)]
\item $f_p=f(\cdotp,p)\colon M\to X$ is a holomorphic Legendrian map for every $p\in P$, 
\item $f_p$ is horizontal for every $p\in Q$, and
\item every vertical map $f_p$ in the family is nondegenerate (see Definition \ref{def:nondegenerate}).
\end{enumerate}
Then, there exists a homotopy $f^t \colon M\times P\to X$ $(t\in [0,1])$ such that 
$f^t_{p} := f^t(\cdotp,p)\in \Lscr(M,X)$ for every  $(p,t)\in P\times [0,1]$ 
and the following conditions hold:
\begin{enumerate}[\rm (1)]
\item $f^t_{p} =f_p$\ \ for every $(p,t)\in (P\times \{0\}) \cup (Q\times [0,1])$, 
\item each vertical map $f^t_{p}$ for $(p,t)\in P\times [0,1]$ is nondegenerate, and
\item $f^1_{p} \in  \Lscr(M,X)$  is horizontal for every $p\in P$. 
\end{enumerate}
\end{theorem}

Theorem \ref{th:B} also holds, with the same proof, if a holomorphic map $M\to X$ is understood as a map
of class $\Ascr^r(M,X)$ for some $r\in \{1,2,\ldots,\infty\}$.
On the other hand, we do not know whether it still holds
if the family $f_p$ includes degenerate vertical curves. The main problem is 
to handle the proof of  Proposition \ref{prop:approx1} with a continuous dependence
on the parameter in a neighbourhood of such curves. 

Theorem \ref{th:B} immediately implies the following result, where $\Lscr_\mathrm{hor}(M,X)$ 
denotes the space of all horizontal holomorphic Legendrian curves $M\to X$. 

\begin{corollary}
The inclusion
\[
	\Lscr_\mathrm{hor}(M,X) \longhookrightarrow 
	\Lscr(M,X)\setminus\{\textrm{degenerate vertical curves}\}
\]
is a weak homotopy equivalence.
\end{corollary}

In the proof of Theorem \ref{th:B} we shall use the following observation.

%
%
\begin{lemma}\label{lem:nonconstant}
If a holomorphic map $f=(z^0,[h_0:\cdots : h_n])\colon M\to X_{z^0}$  is nondegenerate, 
then none of the component functions $h_i$ is identically zero on $M$ and each quotient $h_i/h_j$ 
for $i\ne j$ is nonconstant on $M$.
\end{lemma}

\begin{proof}
Each equation $h_i=0$, or $h_i/h_j=\textrm{constant}$ for $i\ne j$, determines a  projective subspace 
of $\CP^n$, so a map satisfying such an equation is degenerate.
\end{proof}

\begin{remark}\label{rem:nondeg}
We record another observation for future reference. 
Let $U\subset Z$ be a contractible coordinate neighbourhood of a point $z^0\in Z$, and
let $f=(z^0,h)\colon M\to X$ be a vertical holomorphic map from a 
compact bordered Riemann surface. Any holomorphic map $\tilde f\colon M\to X$ sufficiently close to 
$f$ is contained in $X_U$, so it is of the form $\tilde f=(\tilde g,\tilde h)$ with respect 
to the coordinates on $U$, with $\tilde h\colon M\to \CP^n$. If $f$ is nondegenerate 
and $\tilde f$ is close enough to $f$, then the vertical component $\tilde h\colon M\to\CP^n$ 
of $\tilde f$ is also nondegenerate. 
\qed\end{remark} 

\begin{proof}[Proof of Theorem \ref{th:B}]
Let $\pi:X=\P(T^*Z)\to Z$ denote the base projection. Write 
\[
	g_p=\pi \circ f_p: M\to Z,\qquad p\in P,
\] 
and set
\[
	P_0=\{p\in P: f_p\ \ \text{is vertical}\} = \{p\in P: g_p\ \ \text{is constant}\}.
\]
Note that $P_0$ is a compact set disjoint from $Q$.
If $p\in P_0$ and $q\in P$ is close to $p$, then $g_q(M)$ is contained in a coordinate
neighbourhood of the point $g_p(M)$ in $Z$. Recall that the parameter 
space $P$ is assumed to be locally contractible. Hence, by compactness of $Q$ 
there are finitely many pairs of contractible open sets $V'_j\Subset V_j\subset P$ 
and contractible coordinate neighbourhoods $U_j\subset Z$ for $j=1,\ldots,r$ such that 
\[
	g_p(M)\subset U_j\ \ (p\in \overline V_j,\ j=1,\ldots,r), \qquad P_0 \,\subset\, \bigcup_{j=1}^r V'_j. 
\]

Fix an index $j\in\{1,\ldots, r\}$. Let $z=(z_0,\ldots,z_n)$ be holomorphic coordinates  on 
$U_j\subset Z$, and let $[\zeta_0:\cdots:\zeta_n]$ be the corresponding homogeneous coordinates
on $\P(T_{z}^* Z)\cong \CP^n$. The contact structure
on $X_{U_j} =\P(T^* U_j)$ is then given by the kernel of the $1$-form $\eta= \sum_{i=0}^n \zeta_i \, dz_i$
\eqref{eq:eta}. For every $p\in \overline V_j$ we have that 
\[
	f_p=(g_p,\tilde h_p),\qquad g_p:M\to U_j,\qquad \tilde h_p:M\to \CP^n.
\]
Let $\tau\colon \C^{n+1}_*\to \CP^n$ denote the standard projection.
Since $M$ has the homotopy type of a bouquet of circles and the set $V_j\subset P$ 
is contractible, there is a continuous family of holomorphic maps 
\begin{equation}\label{eq:hp}
	h_p=(h_{p,0},\ldots, h_{p,n}) \colon M\to \C^{n+1}_*,\qquad p\in V_j,
\end{equation}
such that $\tau\circ h_p=\tilde h_p$ for all $p\in V_j$
(see \cite[Lemma 5.1]{AlarconForstnericLopez2019JGEA}).

The desired homotopy $f^t_p$ satisfying Theorem \ref{th:B} will be constructed in $r$ steps.
The scheme of proof is similar to the proof of Theorem \ref{th:A}.

Let us consider the initial step with $j=1$. Since the map $\tilde h_p\colon M\to\CP^n$ 
is assumed to be nondegenerate for every $p\in P_0\cap V_1$, 
there exists a neighbourhood $W_1\Subset V_1$ of the compact set $P_0 \cap \overline{V'_1}$ 
such that for every point $p\in W_1$ the map $\tilde h_p$  is nondegenerate
and its lifting $h_p\colon M\to \C^{n+1}_*$  \eqref{eq:hp} 
satisfies the assumptions in Remark \ref{rem:approx1}.
Hence, the remark (which is a parametric version of Proposition \ref{prop:approx1}) 
furnishes a homotopy of Legendrian curves $\tilde f^t_p\in \Lscr(M,X_{U_1})$ $(p\in W_1,\ t\in [0,1])$ 
satisfying the following conditions:
\begin{enumerate}[\rm (i)]
\item $\tilde f^0_p=f_p$, 
\item the map $\tilde f^t_p$ is horizontal and uniformly close to $f_p$ for every $t\in (0,1]$
(also uniformly with respect to $p\in W_1$ and $t\in [0,1]$), and 
\item every map $\tilde f^t_p$ $(p\in W_1,\ t\in [0,1])$ has the same vertical component as $f_p$,
and this component is nondegenerate (by the definition of the set $W_1$).
\end{enumerate}

We now explain how modify the homotopy $\{\tilde f^t_p : p\in W_1,\ t\in [0,1]\}$ 
in order to make it constant (equal to $\{f_p\}$) for the parameter values $p$ near $bW_1$. 
At the same time, we shall take care not to create any degenerate vertical maps 
in the resulting family.

Choose open sets $W'_1,W''_1\subset P$ such that 
\begin{equation}\label{eq:W1}
	 P_0\cap \overline {V'_1}\, \subset\, W'_1 \,\Subset\, W''_1 \,\Subset\, W_1,
\end{equation}
and then choose a compact set $Q_0\subset P$ such that
\begin{equation}\label{eq:Q0}
	Q\subset \mathring Q_0,\qquad P_0\cap Q_0=\varnothing,\qquad 
	P\setminus Q_0 \,\subset\, W'_1 \cup \bigcup_{j=2}^r V'_j.
\end{equation}
Let $\chi\colon P\to [0,1]$ be a continuous function satisfying
\begin{equation}\label{eq:chi}
	\chi=1\ \text{on}\ W'_1 \setminus Q_0  \quad \text{and}
	\quad \chi=0\  \text{on}\ (W_1\cap Q) \cup (W_1\setminus W''_1).
\end{equation}
In particular, we have that $\chi=1$ on $P_0\cap \overline {V'_1}$ 
in view of \eqref{eq:W1} and \eqref{eq:Q0}.
We define a new homotopy $f^t_p\in\Lscr(M,X)$ $(p\in W_1)$ by
\begin{equation}\label{eq:ftp}
	f^t_p= \tilde f^{t \chi(p)}_p,\qquad p\in   W_1, \ t\in [0,1].  
\end{equation}
Clearly, $f^0_p=f_p$ for all $p\in   W_1$.
Note also that the vertical component of $f^t_p$ equals the vertical component of $f_p$ 
by condition (iii) above, i.e., it is independent of $t\in [0,1]$.

Since $\chi=0$ on $W_1\setminus W''_1$ (see \eqref{eq:chi}), 
the homotopy $f^t_p$ is independent of $t$ there, so it extends to all $p\in P$ by setting 
$f^t_p=f_p$ for $p\in P\setminus W_1$ and $t\in[0,1]$. 

For $p\in \overline{W'_1} \setminus Q_0$ we have that $\chi(p)=1$ (see the first condition in \eqref{eq:chi}) 
and hence $f^t_p = \tilde f^t_p$ is horizontal for all $t\in (0,1]$ by condition (ii).
 
For every $p\in W_1\cap Q_0$ the map $f_p$ is horizontal by the definition of $Q_0$.
Assuming as we may that $\tilde f^t_p$ approximates $f_p$ sufficiently closely 
for every $t\in [0,1]$ (see condition (ii)), all maps $\tilde f^t_p$ remain horizontal,
and by \eqref{eq:ftp} the maps $f^t_p$ ($t\in [0,1]$) are also horizontal.

Thus, the curve $f^1_p\in \Lscr(M,X)$ is horizontal for all $p\in Q'_1:=Q_0\cup \overline{W'_1}$. 
It follows that 
\[
	P_1:=\left\{ p\in P : f^1_p\ \text{is vertical} \right\} 
	\,\subset\, P\setminus Q'_1 \,\subset\, \bigcup_{j=2}^r V'_j. 
\]
(In the last inclusion we also used the third condition in \eqref{eq:Q0}.)
Furthermore, the proof shows that any vertical map in the family $f^t_p$ $(p\in W_1,\ t\in [0,1])$ 
is nondegenerate since it equals the vertical component of $f_p$ which is 
nondegenerate by the definition of $W_1$.

In the second step we apply the same argument to the family $f^1_p$ and the set $V_2$.
First, we choose a neighbourhood $W_2 \Subset V_2$ 
of the compact set $P_1 \cap \overline{V'_2}$ such that for every point $p\in W_2$
the map $h^1_p\colon M\to \C^{n+1}_*$  \eqref{eq:hp} (a lifting of 
the vertical component $\tilde h^1_p$ of $f^1_p$) satisfies the assumptions 
in Remark \ref{rem:approx1}. Next, choose open sets $W'_2,W''_2\subset P$ such that 
\[
	 P_1\cap \overline {V'_2} \subset W'_2 \Subset W''_2 \Subset W_2,
\]
and then choose a compact set $Q_1\subset P$ such that
\[
	Q'_1 \subset \mathring Q_1,\qquad P_1\cap Q_1=\varnothing,\qquad 
	P\setminus Q_1 \,\subset\, W'_2 \cup \bigcup_{j=3}^r V'_j.
\]
By following the same arguments as in step 1, we obtain a new 
homotopy  $f^t_p\in \Lscr(M,X)$ $(p\in P,\ t\in[1,2])$ such that $f^2_p$ is horizontal for all 
$p\in Q_2:=Q_1\cup\overline {W'_2}$ and $f^2_p=f^1_p$ for all $p\in P\setminus W_2$,
and the homotopy does not contain any degenerate vertical maps. 

Clearly this construction can be continued inductively.
After $r$ steps of this kind we obtain a family $f^r_p\in \Lscr_\mathrm{hor}(M,X)$ $(p\in P)$ 
of horizontal Legendrian maps such that $f^r_p=f_p$ for all $p\in Q_0$, 
as well as a homotopy from $f_p$ to $f^r_p$. 
Renaming $f^r_p$ to $f^1_p$ we obtain the desired conclusion.
\end{proof}

%
%

\section{The h-principle for closed holomorphic Legendrian curves}
\label{sec:h-principle}

As before, let $X$ be the projectivised cotangent bundle of a complex manifold $Z$ of dimension 
$n+1\geq 2$, endowed with the standard complex contact structure $\xi\subset TX$
(see Sect.\ \ref{sec:intro}).  Let $\pi:X\to Z$ be the base projection.  
Let $M$ be a compact bordered Riemann surface, sitting as a smoothly bounded compact domain in an 
ambient noncompact surface $\widetilde M$.  We denote by $\Oscr(M, X)$ the space  of germs of 
holomorphic maps from open neighbourhoods of $M$ in $\widetilde M$ into $X$. We call such germs 
{\em closed holomorphic curves} in $X$ (modelled on $M$).  
Let $\Lscr(M,X)$ be the closed subspace of Legendrian curves in $\Oscr(M, X)$.

The space $\Oscr(M, X)$ carries the usual colimit topology defined as follows.  
Let $(U_k)_{k\geq 1}$ be a decreasing basis of connected open neighbourhoods of $M$ in 
$\widetilde M$ such that $U_{k+1}$ is relatively compact in $U_k$ for all $k\geq 1$.  
The colimit topology on $\Oscr(M,X)$ is the finest topology that makes all the inclusions 
$j_k:\Oscr(U_k,X) \hookrightarrow \Oscr(M, X)$ continuous.

The following result will be proved at the end of this section. 

\begin{proposition}  \label{p:map-from-compact}
Let $K$ be a sequentially compact first countable space, $X$ be a complex manifold, 
and $f:K\to \Oscr(M, X)$ be a continuous map.  
Then there are an integer $k\geq 1$ and a continuous map $g:K\to \Oscr(U_k, X)$ such that $f=j_k\circ g$.  
\end{proposition}

Taking $K=[0,1]$, it follows in particular that every continuous path in $\Oscr(M, X)$, that is, 
every homotopy of closed holomorphic curves in $X$, is the image of a continuous path in 
$\Oscr(U_k, X)$ for some $k\geq 1$.  The analogous result for $\Lscr(M,X)$ is immediate.

The first main result of this section is the following basic h-principle for holomorphic Legendrian curves 
in $X=\P(T^*Z)$.  Note that $Z$ need not be an Oka manifold.

%
%
\begin{theorem} \label{th:basic-h-principle}
Let $X$ be the projectivised cotangent bundle of a complex manifold $Z$ of dimension at least 2.  
Let $M$ be a compact bordered Riemann surface.  Then, the inclusion
\[ 
	\Lscr(M,X) \hookrightarrow \Oscr(M, X) 
\]
induces a surjection of path components.  In other words, every closed holomorphic curve in $X$ can be deformed to a closed holomorphic Legendrian curve.  The analogous results hold for open curves 
and bordered curves.
\end{theorem}

\begin{remark}  \label{rem:h-principle}
This result is rightly called a basic h-principle.  Indeed, a formally Legendrian holomorphic map $M\to X$ is a 
holomorphic map $f:M\to X$ with a morphism $TM\to\xi$ over it, or in other words, a morphism $TM\to f^*\xi$ over $M$. 
Since $\xi$ is a vector subbundle of $TX$, such a morphism can be scaled down to the zero morphism.  
Therefore the projection $\Lscr_\mathrm{formal}(M,X)\to\Oscr(M,X)$ is a homotopy equivalence.  
This is a general fact for maps directed by a vector bundle.
\qed\end{remark}

\begin{proof}
Let $f\in\Oscr(M,X)$.  We need to show that $f$ can be deformed to a Legendrian curve, so we may 
assume that $f$ is not vertical.  Lift the nonconstant curve $g=\pi\circ f\in\Oscr(M,Z)$ to a horizontal 
Legendrian curve $h\in\Lscr(M,X)$ (see Proposition \ref{prop:summary}(d)).  
Now compare $f$ and $h$.  View them as holomorphic sections of the projectivisation of the pullback 
bundle $g^*(T^*Z)$ on $M$ (see Sect.\ \ref{sec:basic}).  
The pullback bundle is trivial by the Oka-Grauert principle, so we can view $f$ and $h$ as holomorphic
maps to $\C\P^n$ from a connected open neighbourhood $U$ of $M$ 
in $\widetilde M$.  Since $U$ is Stein and $\C\P^n$ is Oka, the inclusion 
$\Oscr(U,\C\P^n)\hookrightarrow \Cscr(U, \C\P^n)$ is a weak equivalence.  Also, $U$ has the 
homotopy type of a bouquet of circles and $\C\P^n$ is simply connected, so $\Cscr(U, \C\P^n)$ 
is path connected.  Hence, $\Oscr(U,\C\P^n)$ is path connected, so the maps $f$ and $h$ can be 
joined by a path in $\Oscr(M,X)$ as required.

For an open Riemann surface $M$, and with $X$ as above, the same argument shows that every 
holomorphic map $M\to X$ can be deformed to a holomorphic Legendrian map.  
(Note that the map $g=\pi\circ f\colon M\to Z$ does not change in the course of
the proof, so we need not assume that $Z$ is an Oka manifold. What is important 
is that the fibres of the projection $X\to Z$ are Oka.  Indeed, our argument is essentially an Oka principle for liftings.)

For a compact bordered Riemann surface $M$, using the Oka principle in 
\cite[Theorem 6.1]{DrinovecForstneric2008FM}  (see also \cite[Theorem 8.12.1]{Forstneric2017E}), 
a similar argument shows that every bordered holomorphic curve can be deformed to a 
bordered holomorphic Legendrian curve.  
\end{proof}

There are technical and possibly also conceptual obstacles  to proving a 1-parametric h-principle
that we are, so far, unable to overcome. One of the principal difficulties is 
the discontinuity of Legendrian liftings of branched curves, described in 
Proposition \ref{prop:singularity}. The situation is substantially simpler if we restrict ourselves 
to holomorphic curves $f:M\to X$ that are \emph{strong immersions}, 
meaning that the projection $\pi\circ f:M\to Z$ is an immersion (so $f$ itself is also an immersion).  
We denote by 
\[
	\Oscr_\mathrm{si}(M,X)
\]
the subspace of $\Oscr(M,X)$ consisting of strong immersions and let 
\[
	\Lscr_\mathrm{si}(M,X)=\Lscr(M,X)\cap\Oscr_\mathrm{si}(M,X).
\]
The following basic and 1-parametric h-principle is the second main result of this section.

%
%
\begin{theorem} \label{th:better-h-principle}
Let $X$ be the projectivised cotangent bundle of a complex manifold $Z$ of dimension at least 2.  
Let $M$ be a compact bordered Riemann surface.  Then the inclusion
\begin{equation} \label{eq:inclusion}
	\Lscr_\mathrm{si}(M,X) \hookrightarrow \Oscr_\mathrm{si}(M, X) 
\end{equation}
induces a bijection of path components.  If $\dim Z\geq 3$, then the inclusion also induces an epimorphism of fundamental groups, but this fails in general if $\dim Z=2$.
\end{theorem}

\begin{remark}\label{rem:strongimmersions}
Assuming that $\dim Z\ge 2$, it follows from the jet transversality theorem that for a generically chosen 
path of holomorphic maps $h_t:M\to X$ $(t\in[0,1])$, the path of projections $\pi\circ h_t:M\to Z$
consists of immersions. This implies that the inclusion 
$\Oscr_\mathrm{si}(M, X) \hookrightarrow \Oscr(M, X)$ satisfies the basic and the 1-parametric 
h-principle, and hence the conclusion of Theorem \ref{th:better-h-principle} also holds
for the inclusion $\Lscr_\mathrm{si}(M,X) \hookrightarrow \Oscr(M, X)$.
\qed\end{remark}

\begin{proof}
Surjectivity at the level of path components is proved by the same argument as in the proof of Theorem \ref{th:basic-h-principle}.  For the 1-parametric h-principle, let $\gamma\colon [0,1]\to \Oscr_\mathrm{si}(M,X)$ be a continuous path joining curves $\gamma(0), \gamma(1)\in \Lscr_\mathrm{si}(M,X)$.  When $\dim Z=2$, the path $\pi_*\circ\gamma$ in 
$\Oscr(M,Z)$ consists of immersions, so it has a unique lifting to a continuous path $\beta$ in $\Lscr_\mathrm{si}(M,X)$ joining $\gamma(0)$ and $\gamma(1)$.  This proves injectivity at the level of path components (but as shown at the end of the proof, $\beta$ need not be homotopic to $\gamma$).  Now let $\dim Z\geq 3$.  We need to deform $\gamma$, with fixed end points, to a path in $\Lscr_\mathrm{si}(M,X)$.

Recall that a Legendrian lifting $h:M\to X$ of a nonconstant holomorphic map $g:M\to Z$ looks like this in local coordinates $z_0,\ldots,z_n$ on $Z$ and the associated fibre coordinates $\zeta_0,\ldots,\zeta_n$ on $T^*Z$, given by the local frame $dz_0,\ldots,dz_n$ (see \eqref{eq:lifting}): 
\[ h=(g_0,\ldots,g_n,[\zeta_0:\cdots:\zeta_n]), \quad
	\textrm{where}\quad \zeta_0 dg_0+\cdots+\zeta_n dg_n=0. \]
Thus $h$ may be identified with a section on $M$ of the projectivisation of a vector subbundle of corank 1 of the pullback bundle $g^*(T^*Z)$ (whether or not $g$ is an immersion).  

By Proposition \ref{p:map-from-compact}, we may view the path $\pi_*\circ\gamma$ as a continuous map $U\times [0,1]\to Z$, holomorphic and immersive along $U$, where $U$ is a connected open neighbourhood of $M$ in $\widetilde M$.  We pull back $T^*Z$ by $\pi_*\circ\gamma$ and obtain a topological vector bundle $E$ of rank $n+1$ over $U\times[0,1]$ with a continuously varying holomorphic structure on the restriction $E_t$ over $U\times\{t\}$ for each $t\in [0,1]$.  By the Oka-Grauert principle, each bundle $E_t$ is holomorphically trivial.  After shrinking $U$, by a theorem of Leiterer \cite[Theorem 2.12]{Leiterer1990}, as a topological vector bundle, holomorphic along $U$, $E$ is isomorphic to the trivial bundle.  In other words, there is a continuous path of holomorphic isomorphisms $E_0\to E_t$, $t\in [0,1]$.

Now consider the subbundle $F$ of $E$ of rank $n$, locally defined by an equation of the form $\zeta_0 dg_0+\cdots+\zeta_n dg_n=0$ as above.  It is, similarly, trivial.  We would like to trivialise the pair $F\subset E$ together.  Considering $E$ as a constant trivial bundle on $U$, we may think of $F$ as a continuous path of holomorphic subbundles of corank 1, defined by a path in $\Oscr(U, \C\P^n)$.  To trivialise $F$ and $E$ together, we need to lift this path by the projection $\mathrm{GL}(n+1,\C)\to\C_*^{n+1}\to \C\P^n$ to, say, the first row.  The projection is the composition of holomorphic fibre bundles with path connected Oka fibres, so it is an Oka map (we do not need to worry about whether it is locally trivial).  Hence, noting that $U$ is homotopy equivalent to a bouquet of circles, a lifting into $\Cscr(U, \mathrm{GL}(n+1,\C))$ exists and can be deformed to a lifting into $\Oscr(U, \mathrm{GL}(n+1,\C))$, as desired.

Our problem may now be reformulated as follows.  We have a path in $\Oscr(U,\C\P^n)$ with end points in $\Oscr(U,\C\P^{n-1})$.  We wish to deform it, with fixed end points, to a path in $\Oscr(U,\C\P^{n-1})$.  For this we need the inclusion $\Oscr(U,\C\P^{n-1}) \hookrightarrow \Oscr(U,\C\P^n)$ to be a 1-equivalence.  Now $\C\P^n$ is $\C\P^{n-1}$ with a $2n$-cell glued to it.  Hence the inclusion $\C\P^{n-1}\hookrightarrow \C\P^n$ is a $(2n-1)$-equivalence; here, $2n-1\geq 3$.  Therefore the inclusion $\Cscr(U,\C\P^{n-1}) \hookrightarrow \Cscr(U,\C\P^n)$ is a 2-equivalence.  By the Oka-Grauert principle, so is the inclusion $\Oscr(U,\C\P^{n-1}) \hookrightarrow \Oscr(U,\C\P^n)$.

Finally, we give an example showing that the inclusion \eqref{eq:inclusion} does not 
induce an epimorphism of fundamental groups in general if $\dim Z=2$.

Let $Z=\C^2$, so $X = \C^2\times \C\P^1$, and consider the following commuting diagram, in which the vertical maps are homeomorphisms, the second row is a retraction diagram, and $\mathscr I(M,Z)$ is the space of holomorphic immersions $M\to Z$.
\[ 
	\xymatrix@C+=1.5cm{
	\Lscr_\mathrm{si}(M,X) \ar@{^{(}->}[r] \ar[d]_{\pi_*} & \Oscr_\mathrm{si}(M,X) \ar[d] & \\
	\mathscr I(M,Z) \ar[r]^{\hspace{-35pt} f\mapsto (f,[f'])} & \mathscr I(M,Z)\times\Oscr(M,\C\P^1) 
	\ar[r]^{\hspace{27pt} \rm proj} & \mathscr I(M,Z)
	} 
\]
We see that the inclusion $\Lscr_\mathrm{si}(M,X) \hookrightarrow \Oscr_\mathrm{si}(M, X)$ induces an epimorphism of fundamental groups if and only if the space $\Oscr(M,\C\P^1)$ (which is path connected) is simply connected.  We claim that it is not if $M$ is an annulus; in fact, invoking Proposition \ref{p:map-from-compact}, we claim that for every open annulus $U_1$ containing $M$, there is a loop in $\Oscr(U_1,\C\P^1)$ that cannot be deformed to a constant loop in $\Oscr(U_2,\C\P^1)$ for any smaller open annulus $U_2$ containing $M$.  This is true because the inclusion $\Oscr(U,\C\P^1) \hookrightarrow \Cscr(U,\C\P^1)$ is a weak homotopy equivalence for every open annulus $U$ by the Oka-Grauert principle, and the space $\Cscr(S^1,S^2)$ is not simply connected since there are continuous maps $S^1\times S^1\to S^2$ that are not nullhomotopic.
\end{proof}

We conclude this section by proving Proposition \ref{p:map-from-compact}.  In what follows, 
$X$ can be an arbitrary complex manifold of positive dimension.

We build on the results in \cite[Sect.~I.7]{GrauertRemmertAS}.  We need to take them a little further.  
With the notation established at the beginning of the section, write $E_k= \Oscr(U_k, X)$ and 
$E= \Oscr(M, X)$.  The inclusions $E_k\hookrightarrow E_{k+1}$ are the restriction maps.  
They are continuous, but in general not open, that is, not inclusions of subspaces.  
Often they will have a dense image.  Each inclusion is \lq\lq vollstetig\rq\rq.\footnote{A map $\phi:S\to T$ 
of topological spaces is vollstetig if for every open subset $V$ of $T$, every point in $\phi^{-1}(V)$ has a 
neighbourhood $U\subset \phi^{-1}(V)$ such that $\phi(U)$ is relatively compact in $V$ 
\cite[p.~22]{GrauertRemmertAS}.  A vollstetig map is continuous and the composition of maps 
that are vollstetig is vollstetig.}  Indeed, by Montel's theorem, each point in $E_k$ has a basis of 
neighbourhoods that are relatively compact in $E_{k+1}$.  (The proof of this observation requires the Docquier-Grauert-Siu tubular neighbourhood theorem 
and triviality of holomorphic vector bundles over an open Riemann surface, so that 
Montel's theorem can be applied; see \cite[Theorems 3.3.3 and 5.3.1]{Forstneric2017E}.)  
The spaces $E_k$ are second countable and Hausdorff (in fact, metrisable).  
Hence, $E$ is what is called a {\em Silva space} in \cite{GrauertRemmertAS}.  In particular, $E$ is 
Hausdorff.  Also, a sequence $(x_\nu)$ converges to a limit $x$ in $E$ if and only if there is $k$ 
such that $E_k$ contains the sequence and the limit and $x_\nu\to x$ in $E_k$.  For all this, 
see \cite[Sect. I.7]{GrauertRemmertAS}.  

\begin{proposition}  \label{prop:1countable}
The space $E= \Oscr(M, X)$ is not first countable and hence not metrisable.
\end{proposition}

\begin{proof}
Consider first the special case when $M$ is the closed unit disc in $\C$ and $X=\C$.
Create a sequence $f_1^1, f_2^1, f_1^2, f_3^1, f_2^2, f_1^3, f_4^1,\ldots$ 
in $E=\Oscr(M,\C)$ by merging the sequences $(f_j^k)_{j\geq 1}$, where $f_j^k(z)=\big[(1+1/k)^{-1}z\big]^j$, 
for $k\geq 1$.  The $k^\textrm{th}$ sequence converges to $0$ locally uniformly on the open disc of radius 
$1+1/k$ (and on no bigger disc), so it converges to $0$ in $E$. If $E$ were first countable, we could extract 
a subsequence from the merged sequence, going to $0$ and having a subsequence in common with each of 
the sequences $(f_j^k)_{j\geq 1}$, $k\geq 1$.  This would contradict the fact, cited above, that a sequence 
$(x_\nu)$ converges to a limit $x$ in $E$ if and only if there is $k$ such that $E_k$ contains the sequence 
and the limit and $x_\nu\to x$ in $E_k$.

Consider now the general case. Let $\psi$ be a 
holomorphic (non-proper) embedding of the open disc of radius $2$ in $\C$ into $X$.
Take the neighbourhoods $U_k$ of $M$ in $ \wt M$
to be smoothly bounded domains with real analytic boundaries. 
Then $U_k$ admits an inner function $\phi_k$, that is, a proper holomorphic map of $U_k$ onto the unit disc
which extends holomorphically to some neighbourhood of the closure $\overline U_k$
(see e.g.~\cite{Ahlfors1950}).
Now proceed as before, taking $f_j^k(z)=\psi\circ \phi_k(z)^j\in X$ for $j, k\geq 1$. 
For each $k$, as $j\to\infty$, $f_j^k$ converges to the constant map $z\mapsto \psi(0)\in X$ on $U_k$, but diverges on any bigger domain. The proof is then completed just as in the special case.
\end{proof}

\begin{proposition}  \label{p:compact}
Let $L$ be a sequentially compact subset of $E$.  Then, for $k$ sufficiently large, $L$ is a subset of $E_k$ 
and is compact in $E_k$.
\end{proposition}

This is proved in \cite{GrauertRemmertAS} for Silva vector spaces only and with compactness in place 
of sequential compactness \cite[Satz I.8.5]{GrauertRemmertAS}.  
The following proof relies on the spaces $E_k$ being second countable.  Since each $E_k$ is metrisable, 
a subset of $E_k$ is compact if and only if it is sequentially compact.

\begin{proof}
First suppose that $L\not\subset E_k$ for all $k$.  Choose $x_k\in L\setminus E_k$ for each $k$.  
The sequence $(x_k)$ in $L$ has a subsequence converging in $E$ (to a limit in $L$), so the subsequence 
must be contained in $E_m$ for some $m$, which is absurd.  For convenience, assume that $L\subset E_1$.

Since the spaces $E_k$ are second countable, each $E_k$ is the increasing union of open subsets 
$V_1^k\subset V_2^k\subset\cdots$ such that $V_j^k$ is relatively compact in $E_{k+1}$ and contained 
in $V_j^{k+1}$ for all $j,k\geq 1$.

Note that $L$ is closed in $E_k$ for each $k$.  Indeed, suppose that $(x_\nu)$ is a sequence in $L$ converging
to a limit $a$ in $E_k$.  Then $x_\nu\to a$ in $E$.  Since $L$ is sequentially compact in $E$, $(x_\nu)$ has a 
subsequence converging in $E$ to a limit $b$ in $L$.  Now $E$ is Hausdorff, so $a=b\in L$.

Suppose that $L$ is not compact in $E_k$ for any $k$.  Then, for each $k$, $L\not\subset V_k^k$, so there 
is $y_k\in L\setminus V_k^k$.  The sequence $(y_k)$ in $L$ has a subsequence -- say itself -- that converges 
in $E$ to a limit $c$ in $L$, so the sequence and the limit $c$ must be contained in $E_m$ for some $m$ 
with $y_k\to c$ in $E_m$.  Then there is $\ell\geq m$ such that $V_\ell^m$ contains $c$ and hence $y_k$ 
for all sufficiently large $k$.  But for such $k$ with $k\geq \ell$, $V_k^k$ contains $y_k$, which is absurd.
\end{proof}

\begin{proof}[Proof of Proposition \ref{p:map-from-compact}]
Let $K$ be a sequentially compact first countable space and $f:K\to E$ be continuous.  
Then the image $f(K)$ is sequentially compact in $E$.  By Proposition \ref{p:compact}, 
there is $k$ with $f(K)$ compact in $E_k$.  Write $g$ for the map $K\to E_k$ induced by $f$.  We need to 
show that $g$ is continuous.  Let $a_\nu\to a$ in $K$.  Every subsequence of the sequence $(g(a_\nu))$ in 
the compact subset $g(K)=f(K)$ of $E_k$ has a subsequence converging in $E_k$ to a limit $b$.  The latter 
subsequence also converges to $b$ in $E$.  Since $g(a_\nu)=f(a_\nu)\to f(a)$ in $E$, and $E$ is Hausdorff, 
$b=f(a)$.  Hence $g(a_\nu)\to f(a)=g(a)$ in $E_k$.
\end{proof}


\subsection*{Acknowledgements}
F.\ Forstneri\v c is supported by the research program P1-0291 and grants J1-7256 
and J1-9104 from ARRS, Republic of Slovenia.  
F.~L\'arusson is supported by Australian Research Council grant DP150103442.  
The work on this paper was begun while Forstneri\v c visited the University of Adelaide in May 2017, 
and was finished when L\'arusson visited the University of Ljubljana in September 2018.
The authors thank the respective institutions for the invitation, hospitality, and partial support.


%


\vskip 5mm

\noindent Franc Forstneri\v c \\
\noindent Faculty of Mathematics and Physics, University of Ljubljana, Jadranska 19, SI--1000 Ljubljana, Slovenia\\
\noindent Institute of Mathematics, Physics and Mechanics, Jadranska 19, SI--1000 Ljubljana, Slovenia\\
\noindent e-mail: {\tt franc.forstneric@fmf.uni-lj.si}

\vskip 4mm

\noindent Finnur L\'arusson \\
\noindent School of Mathematical Sciences, University of Adelaide, Adelaide SA 5005, Australia \\
\noindent e-mail:  {\tt finnur.larusson@adelaide.edu.au}

\end{document}